\documentclass{article}
 
\usepackage{amsmath,amssymb,amsthm}
\usepackage{fullpage}
\usepackage{cite}
\usepackage{palatino}
\usepackage[framemethod=TikZ]{mdframed}

\mdfdefinestyle{MyFrame}{%
    linecolor=orange,
    outerlinewidth=15pt,
    roundcorner=10pt,
    innertopmargin=2\baselineskip,
    innerbottommargin=2\baselineskip,
    innerrightmargin=20pt,
    innerleftmargin=20pt,
    backgroundcolor=blue!70!white}

\newcommand{\C}{\mathbb{C}}
\newcommand{\R}{\mathbb{R}}
\newcommand{\Z}{\mathbb{Z}}

\newtheorem{thm}{Theorem}[section]
\newtheorem{lem}[thm]{Lemma}
\newtheorem{prop}[thm]{Proposition}
\newtheorem{define}[thm]{Definition}

\newtheorem{example}[thm]{Example}
\newtheorem{rem}[thm]{Remark}
\newtheorem{remark}[thm]{Remark}

\renewcommand{\a}{\alpha}
\renewcommand{\b}{\beta}

\newcommand{\e}{\epsilon}

\newcommand{\1}{\mathbf{1}}
\newcommand{\0}{\mathbf{0}}
\newcommand{\av}[1]{\left|{#1}\right|}

\newcommand{\norm}[1]{\left\|{#1}\right\|}

\newcommand{\be}{\begin{enumerate}}
\newcommand{\bi}{\begin{itemize}}
\newcommand{\ee}{\end{enumerate}}
\newcommand{\ei}{\end{itemize}}
\newcommand{\ii}{\item}

\renewcommand{\phi}{\varphi}

\newcommand{\mc}{\mathcal}

\newcommand{\mf}{\mathfrak}
\newcommand{\AG}[2]{\mathsf{A}_{\Gamma}({#1}({#2}))}
\newcommand{\ag}{\mathsf{A}_{\Gamma}}

\newcommand{\Lap}{\mathsf{Lap}}

\newcommand{\Range}{{\mathsf{Range}}}

\newcommand{\G}{\Gamma}

\newcommand{\twist}{\mathsf{Tw}}

\newcommand{\bigcos}[1]{\cos\left(\frac{2\pi}n {#1}\right)}
\newcommand{\mi}{\mathrm{i}}

\title{Synchronization and Stability for Quantum Kuramoto}
\author{Lee DeVille\\Department of Mathematics\\University of Illinois}

\begin{document}
\maketitle

\begin{abstract}
We present and analyze a nonabelian version of the Kuramoto system, which we call the Quantum Kuramoto system.  We study the stability of several classes of special solutions to this system, and show that for certain connection topologies the system supports multiple attractors.  We also present estimates on the maximal possible heterogeneity in this system that can support an attractor, and study the effect of modifications analogous to phase-lag.
\end{abstract}



\section{Introduction}\label{sec:intro}

There is a long history of studying emergent behaviors in collections of coupled oscillators and other complex physical systems.
One of the most famous and well-studied systems of this type is the Kuramoto model~\cite{Kuramoto.75, Kuramoto.book}:
\begin{equation}\label{eq:K}
  \theta_i' = \omega_i + \sum_{j=1}^n \gamma_{ij} \sin(\theta_j-\theta_i).
\end{equation}
where $\omega_i\in \R$, and $\gamma_{ij} = \gamma_{ji} \ge 0$. The original formulation was posed on a continuum, but the closest discrete analogue would be described by taking the graph to be the homogeneous all-to-all graph (all $\gamma_{ij}$ equal) and the  coupling constants  $\omega_i$ random variables with some  fixed distribution.  Since this time, the
Kuramoto model has been a paradigmatic model for systems
exhibiting synchronization, including biological oscillators~\cite{E, Ermentrout.1992, TOR, Hansel.Sompolinsky.92, BS, Kuramoto.91, Kuramoto.Battogtokh.02, Abrams.Mirollo.Strogatz.Wiley.08, Strogatz.Stewart.93, VO1, VO2, MS1, MS2}, related phenomena such as flocking~\cite{HK, HLRS}, and engineered systems~\cite{Sastry.Varaiya.80, Sastry.Varaiya.81, Chopra.Sprong.05, Chopra.Sprong.09, Dorfler.Bullo.2011, Dorfler.Bullo.12, Dorfler.Chertkov.Bullo.13}.  The history is long and detailed, but many reviews exist~\cite{S, Acebron.etal.05, Arenas.etal.08, Dorogovtsev.etal.08, Dorfler.Bullo.14, PRK.book, Sync.book, Winfree.book}.  One recent observation is that when the graph is sparse, the Kuramoto system can support multiple attractors, with the number of attractors being large when there are many oscillators~\cite{Wiley.Strogatz.Girvan.06, DeVille.12, Mehta.etal.14, Mehta.etal.15, DeVille.Ermentrout.16, Delabays.Coletta.Jacquod.16, Delabays.Coletta.Jacquod.17, Ferguson.18}.

In~\cite{Lohe.09, Lohe.10}, Lohe proposed a nonabelian generalization of the Kuramoto model on the matrix groups $U(d), SU(d)$ and discussed many of its synchronization properties.  This was then followed up by a series of insightful papers~\cite{Chi.Choi.Ha.14, Choi.Ha.14, Choi.Ha.15, Choi.Ha.16, Ha.Ryoo.16, Ha.Ko.Ryoo.17} generalizing this model and discussing its dynamical and stationary properties.  A survey of synchronization models, comparing and contrasting classical and quantum synchronization, can be found here~\cite{Ha.Ko.Park.Zhang.16}.  In~\cite{Witthaut.etal.17} the authors study the connection between quantum entanglement and synchronization.  Since this model is an example of a nonabelian model on a Lie group that recovers the Kuramoto model when an abelian group is chosen, we call this system the ``quantum Kuramoto'' system.  

The main results of this paper are two-fold:  first, we study the stability of various types of special solutions for the quantum Kuramoto system and prove that for certain graph topologies the quantum Kuramoto system can simultaneously support multiple attractors; second, we study various effects of inhomogeneities and/or frustrations on the quantum Kuramoto system.

\subsection{Description of Model}\label{sec:description}

Let $G$ be a matrix Lie group (i.e. a closed subgroup of the general linear group over $\R$ or $\C$) with Lie algebra $\mf g$.  We will assume throughout that $G$ has the {\bf Lohe closure property}~\cite{Lohe.09, Ha.Ko.Ryoo.17} that $Z-Z^{-1}\in\mf g$ whenever $Z\in G$. Let $\Gamma$ be an undirected weighted graph with $n$ vertices with edge weights $\gamma_{ij}\ge 0$ (here undirected implies that $\gamma_{ij} = \gamma_{ji}$). Let $f$ be a real analytic function, 
and let $\Omega = \{\Omega_i\} \in \mf g^n$.  Then the {\bf quantum Kuramoto (QK) system}  is the differential equation
\begin{equation}\label{eq:qk}\tag{QK}
  \frac{d}{dt}X_i \cdot X_i^{-1} = \Omega_i +  \frac12\sum_{j=1}^n \gamma_{ij} (f(X_jX_i^{-1})-f(X_iX_j^{-1})).
\end{equation}
(We are slightly abusing notation here by applying $f$ to elements of the Lie group, but we mean this in the standard manner:  if $f(x) = \sum_{p=1}^\infty a_p x^p$, then $f(G) = \sum_{p=1}^\infty a_p G^p$.  Of course, it is not clear {\em a priori} that~\eqref{eq:qk} then defines a flow on $G$, but see Proposition~\ref{prop:closure} below.)
It will be useful for us to define $F\colon G^n \to \mf g^n$, where the components of $F$ are defined as 
\begin{equation}\label{eq:defofFi}
  F_i(X) = \frac12\sum_{j=1}^n \gamma_{ij} (f(X_jX_i^{-1})-f(X_iX_j^{-1})),
\end{equation}
and thus~\eqref{eq:qk} can be written concisely as 
\begin{equation}\label{eq:qk2}
  X_i' X_i^{-1} = \Omega_i + F_i(X).
\end{equation}
Even more concisely, we write $F\colon G^n \to \mf g^n$ and obtain
\begin{equation}\label{eq:defofF}
  X' X^{-1} = \Omega + F(X).
\end{equation}

The first thing to show is that this is a well-defined flow on the Lie group $G$:

\begin{prop}\label{prop:closure}
  Under the assumption that $Z-Z^{-1}\in\mf g$ for any $Z\in G$, the right-hand side of~\eqref{eq:qk} is in $\mf g$, and therefore $X_i' \in T_{X_i}G$ by right multiplication.  Moreover, the flow~\eqref{eq:qk} admits a smooth local solution for any initial condition, i.e. for any $X(0)\in G$, there exists $t^\star>0$ such that there is a smooth one-parameter family $X(t)$, $t\in[0, t^\star)$, that solves~\eqref{eq:qk}.
\end{prop}
 
\begin{proof}
For any $Z\in G$, $Z^p\in G$, and therefore $Z^p - (Z^{-1})^p \in \mf g$.  Then $f(Z) - f(Z^{-1})$ can be written as the difference of two convergent power series expansions.  Of course if $X\in G^n$, then $Z = X_jX_i^{-1}\in G$, and thus $f(X_jX_i^{-1})-f(X_iX_j^{-1})$.  Finally, noting that $\mf g$ is closed under addition, we have $F_i(X)\in \mf g$ for all $X\in G$.  
Also, note that $F$ is smooth, and by standard Existence--Uniqueness arguments we obtain a local smooth solution.
\end{proof}

\begin{remark}
A few remarks here:
\be

\ii We note that $G$ is closed under multiplication, and $\mf g$ is closed under addition, and this is what makes the above work. The trick here is to have a condition where each piece of the right side of~\eqref{eq:qk} lives in $\mf g$, and this is a consequence of a single nonlinear condition on the Lie group, given by the Lohe closure property.

\ii  The Lohe condition might seem restrictive at first glance, but in fact it applies to many of the commonly considered matrix Lie groups.  As is shown in~\cite{Ha.Ko.Ryoo.17}, any matrix Lie group defined by a condition of the form $Z\in G \Leftrightarrow Z^*QZ = Q$ for a fixed invertible $Q$ has the Lohe closure property.  This includes $O(d), O(p,q)$, etc.  It is also evident from the definition that any connected component of a Lie group with Lohe closure also shares this property, and thus $SO(d), SO(p,q)$ do as well.  Other interesting cases include the invertible diagonal matrices, and the unitary upper-triangular matrices.  See~\cite{Ha.Ko.Ryoo.17} for a full list.

\ii We only assert the existence of a {\em local} solution, which raises the natural question of {\em global} solutions.  This question is addressed in~\cite{Ha.Ko.Ryoo.17}, where they show a flow of type~\eqref{eq:qk} that does not possess a global solution (q.v.~Remark 2.3 of~\cite{Ha.Ko.Ryoo.17}).  In fact, it seems likely that whenever $G$ is noncompact, there will be initial data that generates a flow that blows up in finite time.  However, if $G$ is compact then we get global existence for free (moreover, see~\cite[Proposition 2.2]{Ha.Ko.Ryoo.17} for a sufficient condition for global existence). However, the compactness of $G$ leads to another complication, see Section~\ref{sec:stretch} below.

\ee

\end{remark}

\subsection{Example:  the classical Kuramoto model}\label{sec:classical}

We now show that if we choose the commutative Lie group $G=U(1)$, we recover the classical Kuramoto model (thus inspiring our nomenclature of ``quantum Kuramoto'').  If we choose $G = U(1)$, then $\mf g = \mi\R$ and we can parametrize $G = \exp(\mf g)$, i.e.~$\exp\colon\mi\theta \mapsto e^{\mi\theta}$. Choose $\Omega_i = {\mi \omega_i}\in \mf g$. 

Let us first consider the case where $f(x) = x$.  Then~\eqref{eq:qk} becomes
\begin{align*}
  \mi \theta_i' e^{\mi \theta_i} e^{-\mi \theta_i} &= \mi \omega_i +  \sum_{j=1}^n \gamma_{ij}\frac{\left((e^{\mi(\theta_j-\theta_i)})-(e^{\mi(\theta_j-\theta_i)})\right)}2\\
	&= \mi \omega_i +\mi \sum_{j=1}^n \gamma_{ij} \sin(\theta_j - \theta_i),
\end{align*} 
or
\begin{equation}\label{eq:k}
  \theta_i' = \omega_i + \sum_{j=1}^n \gamma_{ij} \sin(\theta_j - \theta_i),
\end{equation}
which is the classical Kuramoto model on the graph $\Gamma$.

If we consider a more general $f(x) =  \sum_{p=1}^\infty a_p x^p$, then~\eqref{eq:qk} becomes
\begin{align*}
  \mi \theta_i' e^{\mi \theta_i} e^{-\mi \theta_i} &= \mi \omega_i +  \sum_{j=1}^n \gamma_{ij}\sum_{p=1}^\infty \frac{a_p}2  \left((e^{\mi(\theta_j-\theta_i)})^p-(e^{\mi(\theta_j-\theta_i)})^p\right)\\
	&= \mi \omega_i +\mi \sum_{j=1}^n \gamma_{ij}\sum_{p=1}^\infty a_p \sin(p(\theta_j - \theta_i)),
\end{align*}
or
\begin{equation}\label{eq:gk}
  \theta_i' = \omega_i + \sum_{j=1}^n \gamma_{ij} \sum_{p=1}^\infty a_p \sin(p(\theta_j - \theta_i )).
\end{equation}

While~\eqref{eq:k} is the most well-known model of this type, generically one considers phase-coupled models of the form
\begin{equation}\label{eq:H}
  \theta_i' = \omega_i + \sum_{j=1}^n \gamma_{ij} H(\theta_j-\theta_i)
\end{equation}
where $H(\cdot)$ is some odd function~\cite{Strogatz.Stewart.93, Collins.Stewart.93, Galan.etal.05}.  We see that the right-hand side of~\eqref{eq:gk} is just the sine Fourier series expansion of $H(\cdot)$ in~\eqref{eq:H}.  In short,~\eqref{eq:qk} recovers the generic phase-coupled model on the torus when we choose $G=U(1)$.


\section{Stability of special solutions}\label{sec:stability}

We are interested in this paper in studying the stability of certain solutions of~\eqref{eq:qk}, which we define now.

\begin{define}
A solution of~\eqref{eq:qk} with $X_i' = 0$ for all $i$ is called a {\bf fixed point} or {\bf stationary solution}.  A solution $\{X_i\}$ where $X_j(t)X_i^{-1}(t) = Z_{ij} \in G$ is constant in time for all $i,j$ is called a {\bf phase-locked} solution.
\end{define}

As is shown in~\cite{Ha.Ko.Ryoo.17}, any phase-locked state is of the form $X_i = Y_i e^{\Lambda t}$, where $\{Y\}$ is a fixed point and $\Lambda$ a fixed element of $\mf g$.  Also note that~\eqref{eq:qk} is right-invariant:  if $\{X_i\}$ is a solution to~\eqref{eq:qk}, and we write $Y_i = X_i Z$ for fixed $Z$, then the flow for $Y_i$ is exactly the same as for $X_i$.  Moreover, we see that if $\Omega_i = 0$ for all $i$, then $X_i \equiv Z$ is a fixed point of~\eqref{eq:qk}.

In this section, we consider the linearization and stability of several classes of solutions.  In Section~\ref{sec:linearization} we compute various formulas for the linearization of~\eqref{eq:qk}, in Section~\ref{sec:sync} we study the stability of sync and near-sync solutions, and in Sections~\ref{sec:twist},~\ref{sec:twist-flip} we study the stability of twist and twist-flip solutions.

\subsection{Linearization}\label{sec:linearization}
 Let $Y = \{Y_i\}$ be a fixed point of~\eqref{eq:qk}.  (As mentioned above, any right-multiplication applied to $Y$ is also a fixed point, so we mean the equivalence class of $Y$ under this action.)  We present two computations here, as each are useful in their own way: in Proposition~\ref{prop:linear} we present a coordinate-free description of the linearization of~\eqref{eq:qk}, and in Proposition~\ref{prop:linearbasis} we present the description of the matrix for this operator given a choice of basis for $\mf g$.

\begin{prop}  \label{prop:linear}
We write $f(x) = \sum_{p=1}^\infty a_p x^p$.  If $Y = \{Y_i\}$ is a fixed point of~\eqref{eq:qk}, then the linearization of the flow around $Y$ is given by the linear operator $\mc{L}_Y\colon\mf g^n \to \mf g^n$ where 
\begin{equation}\label{eq:defofLY}
 \begin{split}
  (\mc L_YQ)_i &= \sum_{j=1}^n{\gamma_{ij}} \mc L_{Y,ij}(Q_j-Q_i),\\
   \mc L_{Y,ij}W &= \frac12\sum_{p=1}^\infty a_p\sum_{q=0}^{p-1}\left((Y_i^{-1}Y_j)^{q+1} W (Y_i^{-1}Y_j)^{p-q-1}+(Y_j^{-1}Y_i)^q W (Y_j^{-1}Y_i)^{p-q}\right).
 \end{split}
\end{equation}
\end{prop}

\begin{prop}\label{prop:linearbasis}
  Let us choose a basis $M_1,\dots, M_{\dim(\mf g)}$ for $\mf g$, and define the constants
  \begin{equation}\label{eq:defofC}
    \mc L_{Y,ij}M_\beta = \sum_{\alpha=1}^{\dim(\mf g)} {C}_{Y,ij,\alpha\beta} M_\alpha.
  \end{equation}
Then define the $n\dim(\mf g)\times n\dim(\mf g)$ matrix $J_Y$ as a $\dim(\mf g)\times \dim(\mf g)$ block matrix.  The blocks are denoted $J_{Y,kl}$ for $k,\ell = 1,\dots, \dim(\mf g)$, and each block is an $n\times n$ matrix.  The coefficients of $J_{Y,kl}$ are defined for $i\neq j$ by
\begin{equation}\label{eq:defofJblock}
 \left(J_{Y,\a\b}\right)_{ij} = {\gamma_{ij}} {C_{Y,ij,\a\b}},
\end{equation}
and the diagonal elements are chosen so that $J_{Y,\alpha\beta}$ has zero row sum.  Then if we write 
\begin{equation*}
  Q_i = \sum_{k=1}^d x_i^k(t) M_k,
\end{equation*}
and $x$ denotes the full vector $\{x_i^k\}$, then 
\begin{equation*}
  x' = J_Y x.
\end{equation*}
\end{prop}

The details of these computations are in Appendix~\ref{app:linearization-computation} below.   We see that the nullspace of this linear system~\eqref{eq:defofLY} is at least $\dim(\mf g)$-dimensional, since the right-hand side is zero whenever we choose $Q_i\equiv Q\in\mf g$.  This, of course, corresponds to the $\mf{g}$-invariance of the original system under right-multiplication.  If $J_Y$ as defined above is negative semi-definite with only $d$ zero eigenvalues, then the equivalence class of solutions containing $Y$ is asymptotically stable under~\eqref{eq:qk}, and if $J_Y$ has positive eigenvalues, then this equivalence class is unstable  under~\eqref{eq:qk}.

\begin{rem}

Notice that the stability of the point $Y$ does not explicitly depend on the forcing $\Omega$.  Of course, the fixed points of~\eqref{eq:qk}, if they exist, are themselves a function of $\Omega$, so there is an implicit dependence.

\end{rem}

\begin{rem}
  A  restatement of Proposition~\ref{prop:linearbasis} is that once we choose a basis for $\mf g$, the matrix representation of the Jacobian can be written as a ``block Laplacian'' matrix:  there is a $\dim(\mf g)\times\dim(\mf g)$ matrix of blocks, and each of these blocks is an $n\times n$ Laplacian matrix. We also observe that these Laplacian blocks are always subordinate to the weight graph $\Gamma$, i.e. if $\gamma_{ij} = 0$, then the $(i,j)$th entry of each of these Laplacian blocks is also zero.
\end{rem}

\begin{example}
If we choose $f(x) = x$, then Proposition~\ref{prop:linear} simplifies to
\begin{equation}\label{eq:L1}
 \mc L_{Y,ij}Q =   \frac12((Y_i^{-1}Y_j)Q + Q(Y_j^{-1}Y_i)),
\end{equation}
and the coefficients are given (implicitly) by
\begin{equation}\label{eq:C1}
  \sum_{\alpha=1}^{\dim(\mf g)} {C_{Y,ij,\alpha\beta}} M_\alpha = \frac12((Y_i^{-1}Y_j)M_\beta + M_\beta(Y_j^{-1}Y_i)).
\end{equation}
\end{example}

\subsection{Sync and near-sync solutions}\label{sec:sync}

In the case where we choose $\Omega_i\equiv 0$ in~\eqref{eq:qk}, the solution $Y_i \equiv 0$ is a fixed point for the dynamics.  We show below in Proposition~\ref{prop:sync} that this fixed point is stable under some mild assumptions.  Moreover, if we choose $\Omega_i$ small enough, then there is a unique solution to~\eqref{eq:qk}; we show this and give a formula for it in Proposition~\ref{prop:near-sync}.

\begin{define}\label{def:defofLaplacian}
Let $\Gamma$ be a weighted graph, then we define the {\bf graph Laplacian of $\Gamma$}, denoted $\Lap(\G)$, as the matrix 
\begin{equation}\label{eq:defofLG}
  \Lap(\G)_{ij} = \begin{cases} \gamma_{ij},&i\neq j,\\ -\sum_{k\neq i}\gamma_{ik},&i=j.\end{cases}
\end{equation}

We also define 
\begin{equation}\label{eq:defofgn0}
\mf g^n_0 := \left\{(Q_1,\dots, Q_n)\in \mf g^n\bigg| \sum_{i=1}^n Q_i = 0\right\}.
\end{equation}
\end{define}

\begin{remark}  
We can think of the Laplacian as a map from $\R^n$ to itself, or as a map from $\mf g^n$ to itself.

As a map from $\R^n$ to itself, we can see directly that $\Lap(\G)\1=\0$ so that the Laplacian is never invertible, and moreover it is classically known that if $\gamma_{ij}\ge 0$, the multiplicity of the zero eigenvalue is the same as the number of connected components of the graph $\G$.  In particular, the zero eigenvalue is simple iff the graph is connected.

As a map from $\mf g^n$ to itself, a similar argument applies: if we let $Q\1$ denote the constant vector $(Q,Q,\dots,Q)\in \mf g^n$,  then $\Lap(\G)Q\1 = \0$ so that $\Lap(\G)\colon \mf g^n\to\mf g^n$ has at least a $\dim(\mf g)$-dimensional nullspace.  Moreover, if the graph $\G$ is connected, then this is exactly the nullity of this map.

Since $\Lap(\G)$ is a zero-row-sum matrix, it is clear that the $\Range(\Lap(\G))\subseteq \mf g_0^n$ as defined in~\eqref{eq:defofgn0}.  By dimension counting, and the above, it is clear that if $\gamma_{ij}\ge 0$ and  $\G$ is connected, then $\Range(\Lap(\G))= \mf g_0^n$ exactly.  Since $\Lap(\G)$ is symmetric, its kernel and range are orthogonal, and this induces an invertible map from $\mf g_0^n$ to itself.  The inverse of this map is the Moore--Penrose pseudoinverse of $\Lap(\G)$, which we denote as $\Lap^+(\G)\colon\mf g_0^n \to \mf g_0^n$.

Moreover, we note that the linear operator $\mc L_Y$ defined in~\eqref{eq:defofLY} can be considered as an operator with domain and range $\mf g_0^n$ in a similar fashion.
\end{remark}

\begin{prop}\label{prop:sync}
  If $f'(0) > 0$, and  $\Lap(\Gamma)$ is negative semidefinite with a single zero eigenvalue and $\Omega_i\equiv 0$, then the equivalence class of synchronous solutions is asymptotically stable.
\end{prop}

\begin{proof}
 By right-invariance, we can assume that $Y_i\equiv I$, and using Proposition~\ref{prop:linear}, we obtain
 \begin{equation*}
  Q_i' = \sum_{p=1}^\infty p a_p \sum_{j=1}^n \gamma_{ij} (Q_j-Q_i).
\end{equation*}
The first sum is $f'(0)$, and this equation becomes 
\begin{equation*}
  \mc L_IQ = f'(0)\Lap(\Gamma)Q.
\end{equation*}
  Consider any perturbation of $\0$ lying in $\mf g_0^n$.  By the remark above, this perturbation can be written as an (orthogonal) linear combination of eigenvectors of $\Lap(\Gamma)$, and these all have negative eigenvalues, and thus the perturbation decays exponentially fast to zero.
\end{proof}

\begin{prop}\label{prop:near-sync}
  Assume that $\sum\Omega_i = 0$, and write the forcing in~\eqref{eq:qk} as $\epsilon\Omega_i$. Assume again that $f'(0) > 0$, and  $\Lap(\Gamma)$ is negative semidefinite with a single zero eigenvalue. Then there is a near-constant solution $Y_i$ to~\eqref{eq:qk} that is asymptotically stable, i.e. there is a stable fixed point with $Y_jY_i^{-1} = O(\e)$.   Moreover, this solution can be obtained by the formula $Y_i = \exp(\epsilon Q_i)$, where $Q$ is given by
  \begin{equation}\label{eq:nearsync}
    Q = -\Lap^+(\Gamma)\Omega.
  \end{equation}
\end{prop}

\begin{proof}
The map $\mc L_I\colon\mf g_0^n\to\mf g_0^n$ is invertible, and its inverse on the restriction to $\mf g_0^n$ is given by $\Lap^+(\Gamma)$.  By the implicit function theorem,~\eqref{eq:qk} has a solution for $\epsilon$ sufficiently small, and moreover this solution satisfies $\mc L_I Q + \Omega = 0$.  By assumption, $\Omega\in\mf g_0^n$, and solving for $Q$ gives us~\eqref{eq:nearsync}.
\end{proof}

\subsection{Twist solutions and their stability}\label{sec:twist}

We have studied the stability of synchronous solutions in Proposition~\ref{prop:sync}, but what about other solutions to~\eqref{eq:qk}?  All of the results of this section are proved in  Appendix~\ref{app:twist-computation}.

\begin{define}
  Let $T\in G$ satisfy $T\neq I$ and $T^n =I$.  Then $X = \{X_i\}_{i=1}^n$, defined by $X_i = T^i$,
  is called a {\bf twist configuration} in $G^n$.  If it is a solution of~\eqref{eq:qk} we will call it a {\bf twist solution} and sometimes more specifically we call it the {\bf twist solution generated by $T$}.
\end{define}

\begin{define}
The {\bf canonical rotations} in $SO(d)$ are defined as follows:  let $z = \lfloor d/2\rfloor$, and $\theta\in \R^z$.  Then\begin{equation*}
  \twist(\theta) = \twist(\theta_1,\dots,\theta_{z}) := \begin{cases}\displaystyle\bigoplus_{q=1}^{z}\left(\begin{array}{cc}\cos\theta_q&-\sin\theta_q\\\sin\theta_q&\cos\theta_q\end{array}\right), & d\mbox{ is even,}\\ &\\
    \displaystyle\bigoplus_{q=1}^{z}\left(\begin{array}{cc}\cos\theta_q&-\sin\theta_q\\\sin\theta_q&\cos\theta_q\end{array}\right)\oplus I_1& d\mbox{ is odd.}\end{cases}
\end{equation*}
\end{define}

For $T\in SO(d)$, then\cite[Theorem 7.38]{Axler.book} (see also~\cite{Weiner.Wilkens.05}) there is an orthonormal basis for $\R^d$ where $T$ has representation $ \twist(\theta_1,\dots,\theta_{d/2}),$ and (up to this choice of basis) we can think of $T$ as parameterized by these angles $\theta_i$.   If we further assume that $T^n = I$, then this implies that $n\theta_q \in 2\pi \Z$ for all $i$, i.e. that $\theta_q = 2\pi \ell_q/n$, and alternatively we could parameterize this twist by the integers $\ell_q$.  It is standard to refer to $\twist(\theta_1,0,0,\dots,0)$ as a ``single rotation'', and $\twist(\theta_1,\theta_2,0,0,\dots,0)$ as a ``double rotation'', etc., and we will do so here.  We will typically refer to the twist configurations by this integer; for example, ``a single $\ell$-twist'' refers to the rotation $\theta = (2\pi \ell/n, 0,0, \dots, 0)$, a ``double $(\ell_1,\ell_2)$-twist'' corresponds to $\theta = (2\pi \ell_1/n, 2\pi \ell_2/n,0,0,\dots,0)$, etc.  Of course the sync solution is also a twist solution in a trivial way, being a 0-twist.

\begin{define}\label{def:circulant}
  We say that $\Gamma$ is a {\bf symmetric circulant graph} if $\gamma_{ij} = \gamma_{\av{i-j}}$.  We say that $\Gamma$ has  {\bf bandwidth} $K$ if $\gamma_{ij} = 0$ whenever $\av{i-j} > K$ and $\Gamma$ has {\bf strict bandwidth} $K$ if $\gamma_{ij} = 0$ whenever $\av{i-j}\neq K$.   Clearly, $K \le \lfloor n/2\rfloor$.
 \end{define}
 
 \begin{rem}
   We use the standard shorthand of saying that the graph is {\bf nearest neighbor} if it has bandwidth $K=1$, {\bf next-nearest neighbor} if it has bandwidth $K=2$, etc.  We also assume below that at least one of the $\gamma_k > 0$ (thus excluding the trivial uncoupled case $\gamma\equiv 0$).  Also, we will slightly abuse notation and write $\gamma_k$ for $\gamma_{i,i+k}$ (which is by definition independent of $i$). 
 \end{rem}

\begin{prop}\label{prop:twist-solution}
  Let $\Gamma$ be a symmetric circulant graph  and choose $\Omega = 0$ in~\eqref{eq:qk}.  Then every twist configuration is a twist solution.
\end{prop}

The next natural question is which of these twist solutions are stable.  We have several results in this direction.

\begin{thm}\label{thm:2t}
  Single twists are never stable if $\av{\ell} > 1$, and double, triple, etc. twists are never stable.  Stated conversely:  the only possible stable twist solutions are the sync solution, and possibly $\pm 1$-twists.
\end{thm}

From this, we see that there are basically two cases for a coupling graph:  either we have stable $1$-twists (and stable sync solutions), or the only stable twist solution is the sync solution.  In the former case, we will say that the graph ``supports $1$-twists''. 

\begin{thm}\label{thm:stability}
We have the following results for several families of graphs:
\be
\ii {\bf Nearest neighbor.} The nearest-neighbor graph supports $1$-twists. 

\ii {\bf Strict bandwidth.}  If the graph has strict bandwidth $K$, and $K$ does not divide $n$, then the $1$-twist is linearly unstable and thus the graph does not support $1$-twists.    If $K$ does divide $n$, then the $1$-twist is linearly stable.

\ii {\bf Strong local coupling.} Let $K<n/4$.  There is a piecewise linear function $G_{K,n}(\gamma_2,\dots,\gamma_K)$ such that for all $\gamma_1 > G_{K,n}(\gamma_2,\dots,\gamma_K)$, the graph supports $1$-twists.  In other words, as long as we make the nearest neighbor term strong enough, we can guarantee stability of $1$-twists.
\ee
\end{thm}

\begin{remark}
For $K=2$, we see that the function $G_{2,n}(\gamma_2)$ is actually linear, so the theorem reduces to the simpler statement:  the graph supports $1$-twists iff $\gamma_1 / \gamma_2 > \rho^*(n)$ where $\rho^*(n)$ can be computed more or less explicitly.  In particular, $\rho^*(n)$ is zero when $n$ is even, but is positive (and in fact approximately $(\pi/n)^2$ for large $n$) when $n$ is odd (see Appendix~\ref{app:twist-computation}).  
\end{remark}

As in the classical case, there is a restriction on how large the bandwidth can be to give stability (in the previous theorem, we assumed that $K$ is less than $n/4$).  A more general question is how large $K$ can be so that we still obtain stability; of course, this might depend in a complicated manner on the relative sizes of the weights at different distances.  We can simplify this question slightly by assuming all of the weights are equal (without loss of generality setting them all to $1$):

\begin{define}
  Choose $\alpha\in(0,1)$.  For each $n$, let us define the symmetric circulant graph $\Gamma^{(\alpha,n)}$ with edge weights:
  \begin{equation*}
   \gamma_ k = \begin{cases}
   1, & k < \lfloor{\alpha n}\rfloor,\\
   0, & \mbox{else}
   \end{cases}
  \end{equation*}
\end{define}
For each $\alpha$, we call the sequence of graphs $\left\{\Gamma^{(\alpha,n)}\right\}_{n\ge 1}$ the {\bf $\alpha$-sequence}.

\begin{thm}\label{thm:alphastability}
  There is $\alpha^* >0$ (approximately equal to $0.340461$) such that if $0 < \alpha < \alpha^*$, then the $\alpha$-sequence eventually supports $1$-twists, i.e. for any fixed $\alpha$ in this range, there is a $n^* = n^*(\alpha)$ such that for all $n> n^*$, the graph $\Gamma^{(\alpha,n)}$ supports $1$-twists.  Conversely, if $\alpha^* < \alpha \le 1/2$, the $\alpha$-sequence eventually does not support 1-twists.
\end{thm}

\subsection{Explicit formulas for the eigenvalues}

In this section, we give explicit formulas for the spectrum of the linearization around various twist solutions when $\Gamma$ is symmetric and circulant.  We will use these to prove Theorems~\ref{thm:2t},~\ref{thm:stability} and~\ref{thm:alphastability}, but the formulas are of independent interest so we state them here explicitly.

\begin{prop}[Single Rotation]\label{prop:single}
Let  $\Gamma$ be symmetric and circulant. The eigenvalues of the linearization around a single $\ell$-twist are
\begin{equation}\label{eq:defofeigs}
\begin{split}
  \lambda_{\ell,m} &= \sum_{k=1}^K \gamma_k \bigcos{k\ell}\left(\bigcos{km}-1\right),\\
  \mu_{\ell,m} &= \sum_{k=1}^K\gamma_k \left\{ \bigcos{k(\ell+m)}+\bigcos{km} - \bigcos{k\ell}-1\right\},\\
  \nu_{\ell,m} &= \sum_{k=1}^K\gamma_k \left(\bigcos{km}-1\right),
\end{split}
\end{equation}
where $m=0,1,\dots, n-1$.  These are repeated with multiplicities $1, 2(d-2), (d-2)(d-3)/2$ for $\lambda,\mu,\nu$ respectively.
\end{prop}

\begin{remark}
The multiplicities listed above add to $d(d-1)/2$, the dimension of $\mf{so}(d)$, and of course $m$ ranges over $n$ terms as well, so this gives $nd(d-1)/2$ eigenvalues once repeats are counted.  Note also that the multiplicities of individual eigenvalues can be higher, since some of those formulas can repeat, e.g. $\nu_{\ell,m} = \nu_{\ell,-m}$, etc.
\end{remark}

\begin{prop}[Double rotation]\label{prop:double}
Let $\Gamma$ be symmetric and circulant. The eigenvalues of the linearization around a double $(\ell_1,\ell_2)$-twist are
\begin{equation}\label{eq:defofeigs2}
\begin{split}
  \lambda_{\ell_a,m} &= \sum_{k=1}^K \gamma_k \bigcos{k\ell_\alpha}\left\{\bigcos{km}-1\right\},\quad a = 1,2,\\
  \mu_{\ell_a,m} &= \sum_{k=1}^K\gamma_k \left\{ \bigcos{k(\ell_a+m)}+\bigcos{km} - \bigcos{k\ell_a}-1\right\},\quad a = 1,2,\\
  \nu_{\ell_a,m} &= \sum_{k=1}^K\gamma_k \left\{\bigcos{km}-1\right\},\quad a = 1,2,\\
  \kappa^\pm_{\ell_1,\ell_2,m} &= \sum_{k=1}^K \gamma_k \left\{\bigcos{k(\ell_1+m)}+\bigcos{k(\pm \ell_2+m)} - \bigcos{k \ell_1} - \bigcos{\pm k\ell_2}\right\}
\end{split}
\end{equation}
where $m=0,1,\dots, n-1$.  Moreover, the $\lambda$ each have multiplicity $1$, the $\mu$ have multiplicity $2(d-4)$, the $\nu$ have multiplicity $(d-4)(d-5)/2$, and the $\kappa$ have multiplicity 1.
\end{prop}

\begin{prop}[Higher-order Rotations]\label{prop:ho}
Let  $\Gamma$ be symmetric and circulant and $T$ be any twist generated by $\theta$ where at least two of the $\theta_q \neq 0$, e.g. $\theta_a,\theta_b\neq 0$.  Let $\theta_a = (2\pi/n)\ell_a, \theta_b = (2\pi/n)\ell_b$.  Then
\begin{equation}\label{eq:defofeigs3}
  \kappa^{\pm}_{\ell_a,\ell_b,m}= \sum_{k=1}^K \gamma_k \left\{\bigcos{k(\ell_a+m)}+\bigcos{k(\pm \ell_b+m)} - \bigcos{k \ell_a} - \bigcos{\pm k\ell_b}\right\}
\end{equation}
are eigenvalues of the linearization around this twist solution.
\end{prop}

\subsection{Twist-flip solutions}\label{sec:twist-flip}

We can also consider a generalization of twist solutions that we call ``twist-flip'' solutions, which we describe now.  In the interests of brevity, we consider only the nearest neighbor coupling (although there is a similar formalism for general circulant graphs). 

We consider the general solution to a system for nearest-neighbor coupling, i.e. we are looking for a sequence $X_i$ that solve
\begin{equation*}
  0 = (X_{i+1} + X_{i-1})X_i^{-1} - X_i(X_{i+1}^{-1} + X_{i-1}^{-1})
\end{equation*}
Fix $i$, and shift this solution by a right-multiplication so that $X_i = I$.  Writing $A = X_{i+1}X^{-1}_i$, $B = X_{i-1}X_i^{-1}$, then this becomes
\begin{equation*}
  0 = (A+B) - (A^{-1}+B^{-1})
\end{equation*}
or
\begin{equation}\label{eq:ab}
  A-A^{-1} = B^{-1}-B.
\end{equation}
We can see clearly that $B = -A$ or $B=A^{-1}$ are solutions, but they are not the only ones.  If $A,B$ are diagonalizable, then this implies that $A$ and $B$ must commute, since they clearly each commute with each side of~\eqref{eq:ab}. Writing $A = QDQ^{-1}$ and $B = Q\widetilde{D}Q^{-1}$, this means that $D - D^{-1}  = \widetilde{D}^{-1} - \widetilde{D}$.  If we write the eigenvalues of $A$ as $\lambda_i$ and those of $B$ as $\mu_i$, then we have
\begin{equation*}
  \lambda_i - \lambda_i^{-1} = \mu_i^{-1} - \mu_i,
\end{equation*}
or $\mu_i = -\lambda_i,1/\lambda_i$.  In particular, there are many such choices (as we can choose one of two $\mu_i$ independently), and if we make the same choice consistently then we recover the solutions $B=-A, B=A^{-1}$ above.  

If we again consider the case of $G = SO(d)$, and let $A = \twist(\theta_1,\dots,\theta_{d/2})$.  Then the choice of $B=A^{-1}$ rotates in the other direction, whereas $B = -A$ adds $\pi$ to all of the angles.  If we consistently make the choice $B = A^{-1}$ at each point, then this corresponds to a twist solution as before. However, if we choose $B = -A$ at (say) one choice of $X_i$, then we get a ``flip'' across the angle axis. 

\begin{example}\label{exa:doubleflip}
Fix $G = SO(4)$, and let $X_k = \twist(\theta_k, \eta_k)$, where
\begin{equation*}
  \theta_k = k \frac{\pi}{n-2},\quad \eta_k = k\frac{2\pi}n, \quad k=0,1,\dots, n-1.
\end{equation*}
This is not a twist solution, since it doesn't ``wrap around'' all the way:  the angle gap between $X_{n-1}$ and $X_0$ in the first rotation axis is not the same as in all the others.  But, we can see that this will be a twist-flip solution as described above.  For example, if we choose the base point $X_0$, and then $A = X_1, B= X_{n-1}$.  Then we have
\begin{equation*}
  A = \twist(\pi/(n-2),2\pi/n),\quad B = \twist((n-1)\pi/(n-2),2(n-1)\pi/n),
\end{equation*}
and we can check that
\begin{equation*}
  A - A^{-1} = -(B-B^{-1}) = \left(\begin{array}{cccc}0&-2\sin(\pi/(n-2))&0&0\\2\sin(\pi/(n-2))&0&0&0\\0&0&0&-2\sin(\pi/n)\\0&0&2\sin(\pi/n)&0\end{array}\right).
\end{equation*}

\end{example}

We can see from the example above that a large number of choices can be made for these twist-flip solutions.  For example, we can twist in one angle and have a twist-flip in the other (as in the example above), or we could have two twist-flips, but the flips at different places, etc.

\begin{remark}
  The twist-flip solutions are analogous to the saddle points that occur for the classical Kuramoto, see~\cite{DeVille.12}.  Here there is only one angle, but one can twist-flip this angle as well.  The typical saddle point for a circulant coupling that has exactly one unstable eigenvalue is exactly a twist-flip with one flip, e.g.  $\twist(\theta_k)$ in Example~\ref{exa:doubleflip}.  We expect that all twist-flip solutions will be unstable, and with enough work one could likely determine this from the formulas derived above.
\end{remark}

\section{Inhomogenities}

In the previous section, we mostly considered the case of a homogeneous system, i.e. where all of the forcing is zero. In general, there are two ways to add inhomogeneities inspired by extensive studies of the classical Kuramoto.  First, we can consider the case where the $\Omega_i$ are different and ask the basic question:  how different can they be, yet still give a fixed point solution?  We address this question in Section~\ref{sec:stretch}.  We can also consider the generalization to ``frustrated'' systems, where the different nodes are not pushed toward each other, but pushed toward a certain relation.  We consider this question in Section~\ref{sec:frustration}.

\subsection{How ``stretched out'' can the forcing be?}\label{sec:stretch}

The basic question is this: for which choices of $\Omega_i$ does~\eqref{eq:qk} support a fixed point solution?
 Recalling the definition of $F\colon G^n\to \mf g^n$ as defined in~\eqref{eq:defofF},  clearly~\eqref{eq:qk} has a fixed point iff $-\Omega\in \Range(F)$.   What can we say about this range?  Let us sum both sides of~\eqref{eq:qk} over $i$ to obtain
\begin{equation}\label{eq:qksummedi}
  \sum_{i=1}^nX_i' X_i^{-1} = \sum_{i=1}^n \Omega_i +  \frac12\sum_{i=1}^n \sum_{j=1}^n \gamma_{ij} \left(f(X_jX_i^{-1})-f(X_iX_j^{-1})\right).
\end{equation}
If $X$ is a fixed point, then we obtain
\begin{equation*}
   \sum_{i=1}^n \Omega_i = - \frac12\sum_{i=1}^n \sum_{j=1}^n \gamma_{ij} \left(f(X_jX_i^{-1})-f(X_iX_j^{-1})\right).
\end{equation*}
The $f$ term is antisymmetric with respect to $i \leftrightarrow j$ and $\gamma_{ij}$ is symmetric, and thus the double sum gives zero.  This gives the necessary consistency condition $\sum_{i=1}^n  \Omega_i = 0$.  (This corresponds to the well-known necessary condition $\sum_{i=1}^n \omega_i = 0$ in~\eqref{eq:k}.)
Thus $F\colon G^n \to \mf g^n_0$, where we repeat the definition of $\mf g_0^n$ for convenience:
\begin{equation}\tag{\ref{eq:defofgn0}}
\mf g^n_0 := \left\{(Q_1,\dots, Q_n)\in \mf g^n\bigg| \sum_{i=1}^n Q_i = 0\right\}.
\end{equation}

However, note that $F$ is in general not surjective --- in particular, if $G$ is compact, then the image of $F$ is a compact subset of $\mf g^n_0$, but $\mf g^n_0$ is a linear subspace of $\mf g^n$ and thus not compact. 
In general, we should expect that an explicit description of $\Range(F)\subseteq \mf g^n_0$ will be  difficult to obtain.  In fact, it is not completely done even for the classical Kuramoto model (but see~\cite{Bronski.DeVille.Park.12, Dorfler.Chertkov.Bullo.13, Bronski.DeVille.Ferguson.16, Bronski.Ferguson.18, Ferguson.18}), so it is a bit much to hope for a complete description in the more general case without some significant breakthrough.
However, we can obtain circumscribing bounds on the image of $F$ in certain cases.  
\begin{define}
For a fixed matrix norm $\norm\cdot_*$ on $\mf g$, and define the induced $\ell^\infty$ norm on $\mf g^n$ given by 
\begin{equation*}
  \norm{\mathcal{Q}}_{*,\ell^\infty} = \max_{i=1}^n \norm{Q_i}_*.
\end{equation*}
\end{define}

\begin{prop}\label{prop:norm}
If $\norm{f(Z)-f(Z^{-1})}_* \le C$ for all $Z\in G$,  then 
\begin{equation*}
 \norm{F(X)}_{*,\ell^\infty} \le \frac C2 \norm{\Gamma}_{\infty},
\end{equation*}
where
\begin{equation*}
  \norm{M}_{\infty} = \max_{i=1}^n \sum_{j=1}^n \av{M_{ij}}.
\end{equation*}

\end{prop}

\begin{proof}
  This is a direct computation, as:
\begin{align*}
  \norm{F}_{*,\ell^\infty}
  	&= \max_{i=1}^n \norm{F_i(X)}	= \max_{i=1}^n \frac12 \sum_{j=1}^n \av{\gamma_{ij}} \norm{f(X_jX_i^{-1}) - f(X_iX_j^{-1})}_*\\
	&\le \frac C2 \max_{i=1}^n\sum_{j=1}^n \av{\gamma_{ij}} = \frac C2 \norm{\Gamma}_\infty.
\end{align*}  
\end{proof}

The classical version of this estimate is well-known, but in this case we can obtain more information by using different norms on $\mf g$, as we illustrate in the following example.

\newcommand{\lpnorm}[2]{\norm{#1}_{\ell^{#2}}}

\begin{example}\label{exa:so3}
  Let us choose $G = SO(3)$ with $\mf g = \mf s\mf o(3)$.  The map $\exp\colon \mf g\to G$ is surjective, so that every $X\in G$ is $\exp(Q)$ for some $Q\in \mf g$. 
  If we write the generic element of $\mf g$ by
\begin{equation}\label{eq:defofQ}
 Q = \left(\begin{array}{ccc}0&a&b\\-a&0&c\\-b&-c&0\end{array}\right),    
\end{equation}
it is not hard to show that $Q^{2k+1}\in \mf g$ for all $k$ and, moreover,
\begin{equation*}
  \exp(Q) - \exp(-Q) = 2\sinh(Q) = \frac{2\sin(d)}d Q,
\end{equation*}
where $d^2 = a^2+b^2+c^2$.  So, if we make the specific choice $f(x) = x$, then 
\begin{equation*}
  \norm{f(Z)-f(Z^{-1})}_* = 2\frac{\av{\sin(d)}}d \norm{Q}_*.
\end{equation*}
(More generally, if we pick $f(x) = x^p$, then $Z = exp(pQ)$ and we obtain a similar norm, but we stick with $f(x) = x$ for now.)
We can consider the $\ell^p$ family of norms for $\mf{so}(3)$, i.e. if $M$ is a $d\times d$ matrix and $v\in\R^d$, then
\begin{equation*}
  \lpnorm{M}p := \left(\sum_{i,j=1}^d \av{M_{ij}}^p\right)^{1/p},\quad \lpnorm{v}p = \left(\sum_{i=1}^d \av{v_i}^p\right)^{1/p}.
\end{equation*}
From~\eqref{eq:defofQ} we have
\begin{equation*}
  \lpnorm{Q}p = 2^{1/p}\lpnorm{(a,b,c)}p,
\end{equation*}
so
\begin{equation}\label{eq:est1}
  \norm{\exp(Q) - \exp(-Q)}_{\ell^p,\ell^\infty} \le 2^{1+1/p} \av{\sin(d)}\frac{\lpnorm{(a,b,c)}p}{\lpnorm{(a,b,c)}2} \le 2^{1+1/p}\mf C_p,
\end{equation}
where
\begin{equation*}
  \mf C_p = \sup_{x\in \R^3} \frac{\lpnorm x p}{\lpnorm x 2}.
\end{equation*}
Using Proposition~\ref{prop:norm}, this means that a necessary condition for $\Omega$ to support a fixed point for~\eqref{eq:qk} is that for all $1\le p \le \infty$
\begin{equation*}
  \norm{\Omega}_{\ell^p,\ell^\infty} \le 2^{1+1/p}\mf C_p.
\end{equation*}
Writing 
\begin{equation*}
  B_{\ell^p,\ell^\infty}(R):= \{\Omega: \norm{\Omega}_{\ell^p,\ell^\infty} \le R\},
\end{equation*}
we can combine all of these bounds into a single necessary condition for $\Omega$ to support a fixed point, namely:
\begin{equation*}
  \Omega \in \bigcap_{1\le p \le \infty}B_{\ell^p,\ell^\infty}(2^{1+1/p}\mf C_p).
\end{equation*}
\end{example}

\subsection{Frustration operators}\label{sec:frustration}

We can modify the equations~\eqref{eq:qk} to add in (what could be called) frustration operators. Choose $A,B\in G$, and define $g(Z) = f(Z) - f(Z^{-1})$, and then consider the flow 
\begin{equation}\label{eq:qkfrust}
  \frac{d}{dt}X_i \cdot X_i^{-1} = \Omega_i +  \frac12\sum_{j=1}^n \gamma_{ij} (g(BX_jX_i^{-1}A^{-1})-g(BA^{-1})).
\end{equation}
If we choose $A=B=I$, then this recovers~\eqref{eq:qk}.  By arguments similar to those in Section~\ref{sec:description}, this is also a well-defined flow on $G^n$.  If we choose all $\Omega_i \equiv 0$, then the constant solution $X_i \equiv X$ is again a fixed point.

\begin{prop}  \label{prop:linearfrust}
We write $f(x) = \sum_{p=1}^\infty a_p x^p$.  If $Y = \{Y_i\}$ is a fixed point of~\eqref{eq:qkfrust}, then the linearization of the flow around $Y$ is given by the linear operator $\mc{L}_Y\colon\mf g_0^n \to \mf g_0^n$ where 
\begin{equation}\label{eq:defofLYfrust}
 \begin{split}
  (\mc L_YQ)_i &= \sum_{j=1}^n{\gamma_{ij}} \mc L_{Y,ij}(Q_j-Q_i),\\
   \mc L_{Y,ij}W = \frac12\sum_{p=1}^\infty a_p&\left\{(Y_i^{-1}BY_j)\left(\sum_{q=0}^{p-1}(BY_jY_i^{-1}A^{-1})^{q} W (BY_jY_i^{-1}A^{-1})^{p-q-1}\right)(Y_i^{-1}A^{-1}Y_i)\right.\\
   & \left.+(Y_i^{-1}AY_i)\left(\sum_{q=0}^{p-1}(Y_j^{-1}B^{-1}AY_i)^q W (AY_iY_j^{-1}B^{-1})^{p-q-1}\right)(Y_j^{-1}B^{-1}Y_i)\right\}.
 \end{split}
\end{equation}
\end{prop}

We prove this proposition in Appendix~\ref{app:linearization-computation}.  In particular, one thing that we show there is that the right-hand side of~\eqref{eq:defofLYfrust}, while complicated, always lies in $\mf g$.  However, one thing to point out is that the operator is no longer symmetric, and we do not expect to obtain real eigenvalues --- even when we linearize around a sync solution!  

For example, 
  If we consider the case where $f(x) = x$, then~\eqref{eq:defofLYfrust} simplifies to
\begin{equation*}
  \mc L_{Y,ij}W = \frac12\left((Y_i^{-1}BY_j)W(Y_i^{-1}A^{-1}Y_i) + (Y_i^{-1}AY_i)W(Y_j^{-1}B^{-1}Y_i)\right).
\end{equation*}
If we have a sync solution $Y_i \equiv Y$, then without loss of generality we can right-multiply to set $Y_i \equiv I$, and then we obtain
\begin{equation*}
  \mc L_{Y,ij}W = \frac12(BWA^{-1} + AWB^{-1}).
\end{equation*}
This is certainly not a self-adjoint operator in general. Consider a concrete example below:

\begin{example}
  Let us consider $G=SO(4)$, and assume $f(x) = x$ as above.  Choose $A = \twist(2\pi/9,0)$ and $B = \twist(0,2\pi/9)$.  Choose the standard basis $M_{12}, \dots$ for $\mf{so}(4)$ as above, and then we can compute the  transfer operator as in~\eqref{eq:defofC} (noting that the system is independent of $i,j$) as the matrix
  \begin{equation*}
 \frac12\left(
\begin{array}{cccccc}
 \cos(2\pi/9) & 0 & 0 & 0 & 0 & 0 \\
 0 &  \cos ^2\left(\frac{2 \pi }{9}\right)+1 & \cos \left(\frac{2 \pi }{9}\right) \sin \left(\frac{2 \pi }{9}\right) & \cos \left(\frac{2 \pi }{9}\right) \sin \left(\frac{2 \pi }{9}\right) & \sin ^2\left(\frac{2 \pi }{9}\right) & 0 \\
 0 & -\cos \left(\frac{2 \pi }{9}\right) \sin \left(\frac{2 \pi }{9}\right) &  \cos ^2\left(\frac{2 \pi }{9}\right)+1& -\sin ^2\left(\frac{2 \pi }{9}\right) & \cos \left(\frac{2 \pi }{9}\right) \sin \left(\frac{2 \pi }{9}\right) & 0 \\
 0 & -\cos \left(\frac{2 \pi }{9}\right) \sin \left(\frac{2 \pi }{9}\right) & -\sin ^2\left(\frac{2 \pi }{9}\right) &  \cos ^2\left(\frac{2 \pi }{9}\right)+1 & \cos \left(\frac{2 \pi }{9}\right) \sin \left(\frac{2 \pi }{9}\right) & 0 \\
 0 & \sin ^2\left(\frac{2 \pi }{9}\right) & -\cos \left(\frac{2 \pi }{9}\right) \sin \left(\frac{2 \pi }{9}\right) & -\cos \left(\frac{2 \pi }{9}\right) \sin \left(\frac{2 \pi }{9}\right) &  \cos ^2\left(\frac{2 \pi }{9}\right)+1 & 0 \\
 0 & 0 & 0 & 0 & 0 & \cos(2\pi/9) \\
\end{array}
\right)\end{equation*}
This matrix is neither symmetric, nor skew-symmetric.  If we consider the nearest-neighbor coupling, then the Jacobian at this point is just the block matrix where we replace each term in this matrix with an $n\times n$ standard $(1,-2,1)$ Laplacian.  We can compute the eigenvalues of this matrix numerically (e.g. for $n=5$) and see that while the matrix is stable, it has eigenvalues that are not real.

\end{example}

Let us motivate the above by considering the classical analogue. Again choose $G = U(1), \mf g = \mi\R$ where we parametrize $G = \exp(\mf g)$, $\theta \mapsto e^{\mi\theta}$.  Choose $A = e^{\mi \alpha/2}, B=e^{-\mi \alpha/2}$, and write $\Omega_i = {\mi \omega_i}$. If $f(x) = \sum_{p=1}^\infty a_p x^p$, then we obtain
\begin{align*}
   \mi \theta_i' e^{\mi \theta_i} e^{-\mi \theta_i} &= \sum_{p=1}^\infty  \frac{a_p}2\left(e^{\mi p(\theta_j - \theta_i - \alpha)}- e^{-\mi p \alpha}\right)-\left(e^{-\mi p(\theta_j - \theta_i - \alpha)}- e^{\mi p \alpha}\right)\\
	&= \mi \sum_{p=1}^\infty a_p \left(\sin(p(\theta_j - \theta_i - \alpha)) - \sin(p\alpha)\right),
\end{align*}
and~\eqref{eq:qkfrust} becomes
\begin{align*}
  \mi \theta_i' e^{\mi \theta_i} e^{-\mi \theta_i} 
  &= {\mi \omega_i}+\sum_{j=1}^n \gamma_{ij} \sum_{p=1}^\infty a_p \left(\sin(p(\theta_j - \theta_i - \alpha)) - \sin(p\alpha)\right),
\end{align*}
or
\begin{equation}\label{eq:KSp}
  \theta_i' = \omega_i + \sum_{j=1}^n \gamma_{ij} \sum_{p=1}^\infty a_p \left(\sin(p(\theta_j - \theta_i - \alpha)) - \sin(p\alpha)\right).
\end{equation}
The simplest choice to make here is that $a_1 = 1$ and all other $a_p = 0$, which then gives
\begin{equation}\label{eq:KS}
  \theta_i' = \omega_i + \sum_{j=1}^n \gamma_{ij} \left(\sin(\theta_j - \theta_i - \alpha) - \sin(\alpha)\right),
\end{equation}
giving the equations commonly known\footnote{As argued in~\cite{Bronski.Carty.DeVille.17}, it is likely more appropriate to call these the Sakaguchi--Shinomoto--Kuramoto equations.} as Kuramoto--Sakaguchi equations~\cite{Kuramoto.Sakaguchi.1986, Sakaguchi.etal.1987, Sakaguchi.etal.1988, DeSmet.Aeyels.2007, Omelchenko.Wolfrum.2013, Kirkland.Severini.15, Ha.Xiao.15, Bronski.Carty.DeVille.17}.  In this case, the angle $\alpha$ is the ``frustration angle'' that causes angles to not quite synchronize in~\eqref{eq:KS}.

\section{Conclusions and Discussion}\label{sec:outtro}

Many of the above results either compare or contrast to the types of results that exist for the classical Kuramoto system~\eqref{eq:k}; this model has been extensively studied by a large number of authors and many results exist, as cited above.  (Recall that the classical Kuramoto model can be obtained as a specific case of the quantum model by choosing $G = U(1)$ or $G=SO(2)$ --- we prove this claim about $G = U(1)$ in Section~\ref{sec:classical} above, but a similar direct computation shows the same for $G=SO(2)$ with the standard generator.) In particular, we have studied the stability of special solutions to the system~\eqref{eq:qk} and obtained bounds on the heterogeneity of allowable forcing terms that can support fixed point solutions.

In Section~\ref{sec:stability} we study twist and twist-flip solutions in the case of circulant coupling.  These also exist in the classical case~\cite{Wiley.Strogatz.Girvan.06, DeVille.12}, as do their generalizations to more complex graphs.  In particular, it is shown in~\cite{Mehta.etal.14, Mehta.etal.15, DeVille.Ermentrout.16, Delabays.Coletta.Jacquod.16, Delabays.Coletta.Jacquod.17, Ferguson.18} that complex networks with a wide variety of topologies support multiple stable fixed points.  In particular,~\cite{Ferguson.18} shows that a graph can have many cycles with any interconnection pattern one wishes, and this graph will support multiple stable fixed points --- as long as the number of vertices in each cycle are taken large enough.  In the current work, we have only considered the circulant case, but we expect that a generalization of the cited techniques to the quantum case would also show that complex graph topologies can also support multiple stable fixed points.  This is an obvious avenue for future study.

One  stark contrast between the classical and quantum cases is shown in Section~\ref{sec:twist}. In the classical case, only single twists can appear (as there is only one angle).  However, the formulas for a single twist in Proposition~\ref{prop:single} are still valid, but only the $\lambda_{\ell,m}$ appear (the $\mu_{\ell,m},\nu_{\ell,m}$ have zero multiplicities).  In particular, in the classical case, one can obtain stable $\ell$-twists with $\ell>1$, as long as $n$ is large enough.  However, in the quantum case, for $\ell>1$, the $\mu_{\ell,m}$ modes destabilize the system, only allowing for (at most) $\pm1$-twists.  Theorem~\ref{thm:alphastability} is quite analogous to the stability results in~\cite{Wiley.Strogatz.Girvan.06} taken in a similar limit.  Note here that the $\lambda_{1,m}$ modes are the modes that give the critical parameter for $\alpha^*$, meaning that the scaling is essentially the same as seen in the classical case.

A different but equally stark contrast is the surprising divisibility condition that appears in Theorem~\ref{thm:stability}.  In the classical Kuramoto system, if $K<n/4$, then $1$-twists are always stable (the Jacobian is a graph Laplacian with off-diagonal positive entries and is thus negative semi-definite).  In the quantum case, it is possible to choose coupling with bandwidth much less than $n$ and still obtain an unstable 1-twist (e.g. for any odd $n$ and $K=2$, if $\gamma_2/\gamma_1$ is large enough then the 1-twist has a linearly unstable mode).  The instability arises in the $\mu_{1,m}$ family and not the $\lambda_{1,m}$, as it must, since the latter are the eigenvalues in the classical model.

Another similarity between the classical and quantum cases in explored in Section~\ref{sec:stretch}.  There is a zero-sum condition on the forcing to obtain a fixed point in both the classical and quantum cases.  The result of Proposition~\ref{prop:norm} is exactly analogous to bounds given in the classical case in, for example,~\cite{Dorfler.Bullo.2011, Dorfler.Bullo.12, Dorfler.Chertkov.Bullo.13} --- basically, the forcing terms have a maximal stretch before the fixed point no longer exists.  An interesting complication in the quantum case is the fact that we can use different norms for the forcing terms, as explored in Example~\ref{exa:so3} --- this leads to even more complicated stability boundaries than is found for the classical case (being, as in Example~\ref{exa:so3}, an intersection of an infinite family of open balls). Finally, in Section~\ref{sec:frustration} we show that there are quantum analogues to the Kuramoto--Sakaguchi system.  In both the classical and quantum cases, it is notable that the Jacobian loses symmetry.

\section{Acknowledgments}\label{sec:ack}

The author thanks Jared Bronski, Thomas Carty, and Eddie Nijholt for illuminating discussions in the course of writing this manuscript.  The author would also like to thank an anonymous referee for suggesting a line of investigation that culminated in the entirely new Theorem~\ref{thm:alphastability} and in enhancements to the conclusions of Theorem~\ref{thm:stability}.

%

\appendix

\section{Details for the linearization computation}\label{app:linearization-computation}

If we take Proposition~\ref{prop:linearfrust} and plug in $A=B=I$, then we obtain Proposition~\ref{prop:linear}.  So in fact we will prove Proposition~\ref{prop:linearfrust} directly and consider Proposition~\ref{prop:linear} as a corollary.  First a lemma:

\begin{lem}\label{lem:Lohe}
 Let us assume that $G$ has the Lohe property.  Then the infinitesimal version of the Lohe property is:  for any $Y,Z\in G$, $Q\in \mf g$, we have $YQZ + Z^{-1}QY^{-1} \in \mf g$.
\end{lem}
 
\begin{proof}
  Let us write $X = Ye^{\e Q}Z$ with $Y,Z\in G$ and $Q\in \mf g$.  Then $X\in G$, and by the Lohe property, $X-X^{-1} \in \mf g$.  But then
  \begin{equation*}
  \begin{split}
    X - X^{-1} &= Y(I+\e Q)Z - Z^{-1}(I-\e Q)Y^{-1} \\&= (YZ-Z^{-1}Y^{-1}) + \e(YQZ + Z^{-1}QY^{-1}) + O(\e^2).
\end{split}
  \end{equation*}
  Since $X-X^{-1}\in \mf g$ and $YZ-Z^{-1}Y^{-1} \in \mf g$, we have $YQZ + Z^{-1}QY^{-1}\in \mf g$.
\end{proof}

{\bf Proof of Proposition~\ref{prop:linearfrust}.}  If we expand~\eqref{eq:qkfrust}, then we have
\begin{equation}\label{eq:qkfrust2}
  \frac{d}{dt}X_i \cdot X_i^{-1} = \Omega_i +  \frac12\sum_{j=1}^n \gamma_{ij} \left(f(BX_jX_i^{-1}A^{-1}) - f(AX_iX_j^{-1}B^{-1}) - f(BA^{-1})-f(AB^{-1})\right).
\end{equation}
Assume that $Y = \{Y_i\}$ is a fixed point, which means
\begin{equation}\label{eq:fixedpointfrust}
\begin{split}
  0&=\sum_{j=1}^n \gamma_{ij} \left(f(BY_jY_i^{-1}A^{-1}) - f(AY_iY_j^{-1}B^{-1}) - f(BA^{-1})-f(AB^{-1})\right). \end{split}
\end{equation}
Assume that $f(x) = x^p$, and the general case will follow by linearity.  Writing
\begin{equation*}
  X_i(t) = Y_i \exp(\e Q_i(t)) = Y_i(I + \e Q_i(t)) + O(\e^2),
\end{equation*}
then the left-hand side of~\eqref{eq:qk} (or~\eqref{eq:qkfrust}) is
\begin{equation}\label{eq:lhs}
  Y_i(\e Q_i')(I - \e Q_i)Y_i^{-1} = Y_i Q_i' Y_i^{-1} + O(\e^2).
\end{equation}
We compute
\begin{align*}
  (BX_j X_i^{-1}A^{-1})^p 
  	&= (BY_j(I+\e Q_j)(I-\e Q_i)Y_i^{-1}A^{-1})^p \\
	&= (BY_j Y_i^{-1}A^{-1} + \e BY_j (Q_j-Q_i) Y_i^{-1}A^{-1} + O(\e^2))^p\\
	&= (BY_jY_i^{-1}A^{-1})^p \\&\quad+ \e \sum_{q=0}^{p-1} (BY_jY_i^{-1}A^{-1})^q (BY_j(Q_j-Q_i)Y_i^{-1}A^{-1})(BY_jY_i^{-1}A^{-1})^{p-q-1}.
\end{align*}
The last term comes from noting that we can ignore the term of $O(\e^2)$, and consider only the binomial.  Then every term of $O(\e)$ in the expansion comes from taking $q$ powers of the $O(1)$ term, one power of the $O(\e)$ term, and then $p-q-1$ powers of the $O(1)$ term.  We can reorder these terms slightly, by peeling off one $BY_j$ from the front and one $Y_i^{-1}A^{-1}$ from the back, and we can write
\begin{equation*}
  (BX_j X_i^{-1}A^{-1})^p = (BY_jY_i^{-1}A^{-1})^p + \e \sum_{q=0}^{p-1} (BY_j)(Y_i^{-1}A^{-1}BY_j)^q (Q_j-Q_i)(BY_jY_i^{-1}A^{-1})^{p-q-1}(Y_i^{-1}A^{-1}).
\end{equation*}
Similarly, we obtain
\begin{equation*}
  (AX_i X_j^{-1}B^{-1})^p = (AY_iY_j^{-1}B^{-1})^p - \e \sum_{q=0}^{p-1} (AY_i)(Y_j^{-1}B^{-1}AY_i)^q (Q_j-Q_i)(AY_iY_j^{-1}B^{-1})^{p-q-1}(Y_j^{-1}B^{-1}).
\end{equation*}

Using~\eqref{eq:fixedpointfrust}, we see that when we take the $\gamma$ sum the $O(1)$ term disappears, and then we obtain
\begin{align*}
  \e Y_i Q_i' Y_i^{-1}
  	&=  \frac\e 2\sum_{j=1}^n \gamma_{ij}\sum_{q=0}^{p-1} (BY_j)(Y_i^{-1}A^{-1}BY_j)^q (Q_j-Q_i)(BY_jY_i^{-1}A^{-1})^{p-q-1}(Y_i^{-1}A^{-1})\\
	&+\quad \frac\e2\sum_{j=1}^n \gamma_{ij}\sum_{q=0}^{p-1} (AY_i)(Y_j^{-1}B^{-1}AY_i)^q (Q_j-Q_i)(AY_iY_j^{-1}B^{-1})^{p-q-1}(Y_j^{-1}B^{-1}).
\end{align*}  
Solving for $Q_i'$ gives
\begin{align*}
   Q_i'
  	&= \frac12\sum_{j=1}^n \gamma_{ij}\sum_{q=0}^{p-1} (Y_i^{-1}BY_j)(Y_i^{-1}A^{-1}BY_j)^q (Q_j-Q_i)(BY_jY_i^{-1}A^{-1})^{p-q-1}(Y_i^{-1}A^{-1}Y_i)\\
	&+\quad \frac12\sum_{j=1}^n \gamma_{ij}\sum_{q=0}^{p-1} (Y_i^{-1}AY_i)(Y_j^{-1}B^{-1}AY_i)^q (Q_j-Q_i)(AY_iY_j^{-1}B^{-1})^{p-q-1}(Y_j^{-1}B^{-1}Y_i).
\end{align*}

We now use Lemma~\ref{lem:Lohe}.  If we pair the $q$th term in the first sum with the $p-q-1$st term in the second, they are of the form $\Upsilon(Q_j-Q_i)\Xi + \Xi^{-1}(Q_j-Q_i)\Upsilon^{-1}$, and this term is in $\mf g$.  Therefore the entire right-hand side of~\eqref{eq:defofLYfrust} is also in $\mf g$ by linearity.  This recovers~\eqref{eq:defofLYfrust}.
\qed

\section{Details regarding the twist solution and its stability}\label{app:twist-computation}

There are several results in the body of the paper whose proofs were promised here, namely:  Proposition~\ref{prop:twist-solution}, Theorem~\ref{thm:2t}, Theorem~\ref{thm:stability},
Theorem~\ref{thm:alphastability}, Proposition~\ref{prop:single}, Proposition~\ref{prop:double}, and Proposition~\ref{prop:ho}.  Before we proceed with these proofs, we have three lemmas which we state and prove first, and then proceed with the proofs of the remaining results.

\begin{lem}\label{lem:f}
Let us assume that $0 < \ell < n/K$.  Define 
  \begin{equation*}
    f_{k, \ell}(x) =  \left\{ \bigcos{k(\ell+x)}+\bigcos{kx} - \bigcos{k\ell}-1\right\},
  \end{equation*}
  and then $f_{k,\ell}(x)$ is periodic of period $n/k$.  On the fundamental domain, it is zero at $x=0, n/k-\ell$, negative for $x\in(0,n/k-\ell)$, and positive for $x\in(n/k-\ell,n/k)$.
\end{lem}

\begin{proof}
  The statement of periodicity is straightforward, and we can see the zeros of the function by plugging in.  If we define $g_{k,\ell}(x)$ by 
    \begin{equation*}
    g_{k, \ell}(x) =  \left( \bigcos {k(\ell+x)} + \bigcos{kx}\right),
  \end{equation*}
then this function is (up to a negative multiplicative constant) its own second derivative, meaning that it is concave down when positive and concave up when negative.  Moreover, we see that this function has exactly two roots in the fundamental domain, at
$x = -\frac\ell2 \pm \frac{n}{4k}$.  Finally, note that $f_{k,\ell}(x)$ is just $g_{k,\ell}(x)$ shifted (down) by a constant, and therefore $f_{k,\ell}(x)$ has at most two roots in the fundamental domain.  Thus the statement about the signs on the intervals follows. 
  \end{proof}

\begin{lem}\label{lem:A+iB}
Let $A$ be symmetric, and $B$ skew-symmetric, and define
\begin{equation}
  C = \left(\begin{array}{cc}A&-B\\B&A\end{array}\right).
\end{equation}
Then the eigenvalues of $C$ are the eigenvalues of $A+\mi B$, each repeated twice.  If $v$ is an eigenvector of $A+\mi B$ with eigenvalue $\omega$, then $(\Re(v),\Im(v))$ and $(\Re(v),-\Im(v))$ are eigenvectors of $C$ with eigenvalue $\omega$. 
\end{lem}

\begin{remark}
 Under the assumptions, $C$ is real symmetric and $A+\mi B$ is Hermitian, so the eigenvalues are real.  Moreover, the eigenvalues of $A-\mi B$ are the same as $A+\mi B$.
\end{remark}

\begin{proof}
  Let $(A+\mi B) v = \omega v$.  Taking adjoints also gives $(A-\mi B)v=\omega v$.  Expanding this last gives
\begin{equation*}
  A\Re v + B \Im v + \mi (-B\Re v + A \Im v) = \omega \Re v + \mi \omega \Im v.
\end{equation*}
From this it follows that 
\begin{equation*}
  \left(\begin{array}{cc}A&B\\-B&A\end{array}\right)\left(\begin{array}{c}\Re v \\\Im v\end{array}\right) = \left(\begin{array}{c}A\Re v + B \Im v\\-B\Re v + A \Im v\end{array}\right) =\omega\left(\begin{array}{c}\Re v \\\Im v\end{array}\right),
\end{equation*}
and using the adjoint equation gives the same result for $(\Re v,-\Im v)$. 
\end{proof}

\begin{lem}\label{lem:ABC}
  Let $A$ be symmetric and $B,C$ skew-symmetric $n\times n$ matrices.  Define
  \begin{equation*}
    Q = \left(\begin{array}{cccc}A&B&C&0\\-B&A&0&C\\-C&0&A&B\\0&-C&-B&A\end{array}\right).
  \end{equation*}
  Then the eigenvalues of $Q$ are the eigenvalues of $A\pm\mi (B\pm C)$, or, equivalently, the eigenvalues of $A+\mi(B\pm C)$ repeated twice. 
  
   \end{lem}

\begin{proof}  
Since $Q$ is symmetric, its eigenvalues are all real.   Since $A$ is symmetric and $B,C$ are antisymmetric, either of the two matrices $A+\mi (B\pm C)$ are Hermitian.  Note also that $A\pm\mi (B+C)$ have the same eigenvalues, since if
\begin{equation*}
  (A+\mi(B+C))z = \mu z,
\end{equation*}
then
\begin{equation*}
  (A-\mi (B+C))\overline z = \mu \overline{z}.
\end{equation*}
(Also note that the eigenvalue counts match up; since there are two choices of sign in $A\pm\mi(B+C)$, this gives 4 lists of length $n$, or $4n$ total eigenvalues.)

We compute:
\begin{equation*}
 (A+\mi(B+C))(x+\mi y) = (Ax-(B+C)y)+\mi((B+C)x+Ay),
\end{equation*}
and if $\mu$ is a real eigenvalue of $A+\mi(B+C)$, then we can separate real and imaginary parts as 
\begin{equation*}
  Ax - (B+C)y = \mu x,\quad (B+C)x + Ay = \mu y.
\end{equation*}
We also can compute directly that
\begin{align*}
   Q(x,-y,-y,-x)^t &= (Ax - (B+C)y,-(B+C)x-Ay,-(B+C)x-Ay,-Ax+(B+C)y)^t\\
   	&=\mu(x,-y,-y,-x)^t,
 \end{align*}
and we have shown that $x+\mi y$ is an eigenvector of $A+\mi(B+C)$, iff $(x,-y,-y,-x)$ is an eigenvector of $Q$ with the same eigenvalue.  Similar computations show that
\begin{align*}
  (A+\mi(B+C))(x-\mi y) = \mu(x-\mi y) & \iff Q(-x,-y,-y,x)^t=\mu(-x,-y,-y,x)^t,\\
  (A+\mi(B-C))(x+\mi y) = \mu(x+\mi y) & \iff Q(x,-y,y,x)^t=\mu(x,-y,y,x)^t,\\
  (A+\mi(B-C))(x-\mi y) = \mu(x-\mi y) & \iff Q(-x,-y,y,-x)^t=\mu(-x,-y,y,-x)^t,
\end{align*}
and we are done.
  
\end{proof}

{\bf Proof of Proposition~\ref{prop:twist-solution}.}
We prove this in the case where  $f(x) = x^p$ for $p\ge 1$ and the rest follows by linearity.  If we have
\begin{equation*}
  X_i'X_i^{-1} = \sum_{j=1}^n \gamma_{ij}((X_jX_i^{-1})^p - (X_iX_j^{-1})^p),
\end{equation*}
then plug in $Y_i = T^i$ into the right-hand side, this gives
\begin{equation*}
   \sum_{j=1}^n \gamma_{ij}(T^{p(j-i)} - T^{p(i-j)}).
\end{equation*}
There are two cases: either $n$ is odd or $n$ is even.  If $n$ is odd, then we obtain
\begin{equation}\label{eq:odd}
  \sum_{k=0}^{\lfloor n/2 \rfloor} \gamma_{i,i+k}(T^{pk} - T^{-pk}) + \gamma_{i,i-k} (T^{-pk}-T^{pk}).
\end{equation}
If $n$ is even, then we obtain
\begin{equation}\label{eq:even}
  \sum_{k=0}^{ n/2-1} \gamma_{i,i+k}(T^{pk} - T^{-pk}) + \gamma_{i,i-k} (T^{-pk}-T^{pk}) + \gamma_{i,i+n/2}(T^{pn/2}-T^{-pn/2}).
\end{equation}
Since $T^{pn} = (T^n)^p = I^p = I$, we have $T^{pn/2}T^{pn/2} = I$ and thus $T^{pn/2} = T^{-pn/2}$, so that the last term of~\eqref{eq:even} is zero and thus~\eqref{eq:even} is the same as~\eqref{eq:odd}.  Since $\gamma_{i,i+k} = \gamma_{i,i-k}$ for all $i,k$, then all terms in~\eqref{eq:odd} cancel.  
\qed
  
{\bf Proof of Theorem~\ref{thm:2t}.}
Let us first consider a single $\ell$-twist with  $\av\ell > 1$, and then we will show that $\mu_{\ell,m}$ is positive for some $m$.   Since $\mu_{\ell,m}$ in~\eqref{eq:defofeigs} is invariant under the change $\ell\mapsto -\ell, m\mapsto -m$, we can assume without loss of generality that $\ell >0$.      

Assume then that $\ell > 1$.  From Lemma~\ref{lem:f}, the function $f_{k,\ell}(x)$ is periodic with period $n/k$ and positive on $(n/k-\ell,n/k)$, and from this it implies that it is positive on the domain $(n-\ell,n)$ for every $k$.  This interval has width $\ell$ and thus contains some integers (e.g. $n-1$ is always in the interior of this interval).  In summary, if $\ell > 1$, then $f_{k,\ell}(n-1) > 0$.  But now notice that $f_{k,\ell}(m)$ is just $\mu_{\ell,m}$ in~\eqref{eq:defofeigs2} with $\gamma$ chosen where $\gamma_k = 1$ and other $\gamma_j = 0$.  Since~\eqref{eq:defofeigs} is linear in the components of $\gamma$, and the argument above is independent of $k$, this implies that for any choice of $\gamma$, for a single rotation with $\av\ell >1$, there is a linearly unstable mode.

Now consider a rotation around multiple axes, indexed by $\ell_a$.  If we have $\av{\ell_a}>1$ for any $\ell_a$, then the argument above also applies to $\mu_{\ell_a,m}$ in~\eqref{eq:defofeigs2}.  Thus the only remaining case is if we have multiple nonzero $\ell_a,\ell_b$ with $\av{\ell_a} = \av{\ell_b}= 1$.  As before, since $\kappa^\pm_{\ell_1,\ell_2,m}$ is invariant under the transformation $\ell_1\mapsto-\ell_1,\ell_2\mapsto-\ell_2,m\mapsto-m$, we can assume that $\ell_1 = 1$.  Then simplifying~\eqref{eq:defofeigs3} gives
\begin{equation*}
  \kappa^{+}_{1,1,m} = \kappa^{-}_{1,-1,m} =  2\left(\bigcos{k(m+1)} - 2\bigcos{k}\right)
\end{equation*}
and this is clearly positive at $m=n-1$. \qed

{\bf Proof of Theorem~\ref{thm:stability}.}  We first prove that if $\gamma = \delta_K$ and $K$ does not divide $n$, then one of the $\mu_{1,m}$ in~\eqref{eq:defofeigs} is positive.  Consider again Lemma~\ref{lem:f}: the function $f_{K,1}(x)$ is positive on any interval of the form $(qn/K-1,qn/K)$ with $q\in\Z$. If $n/K$ is not an integer, then $qn/K$ is not an integer for some $q\in \Z$, and therefore $f_{K,1}(x)$ is positive on an integer, meaning that one of the $\mu_{1,m} > 0$.  Similarly, if $K$ does divide $n$, then all of the $\mu_{1,m} \le 0$, since all of the intervals will have width one and integral endpoints.

Now assume $K=1$.  In this case, we see that the function $f_{1,1}(x)$ is positive only on the interval $(n-1,1)$, and therefore $\mu_{1,m} \le 0$ for all $m$, and strictly negative for $m\neq 0,-1$.  

Now let us consider a general choice of $\gamma_k$ with $K< n/4$, and consider the formula for $\lambda_{1,m}$ in~\eqref{eq:defofeigs}.  Let us write this in the form $\sum_{k=1}^K \gamma_k \alpha_{k,m,n}$.  We see by inspection that $\alpha_{k,0,n} = 0$ and therefore $\lambda_{1,0} = 0$ independent of $\gamma_k$.  However, if $k < n/4$ and $m \neq 0$, then $\alpha_{k,m,n}\le 0$, and moreover if $m\neq 0$, $\alpha_{1,m,n} < 0$.  Therefore $\lambda_{1,m} < 0$ if $\gamma_1 > 0$.

Now consider the formula for $\mu_{1,m}$ in~\eqref{eq:defofeigs}.  Let us write this in the form $\sum_{k=1}^K \gamma_k \beta_{k,m,n}$.  We see that $\beta_{k,m,n} = 0$ if $m=0,n-1$, so this sum is again zero.  As mentioned above, it is possible for $\beta_{k,m,n}$ to be positive (in particular it will be for some $m$ if $k$ does not divide $m$).  However, we know that $\beta_{1,m,n} < 0$ for any $n$ and $m\neq 0,n-1$, so therefore this sum has a negative coefficient for $\gamma_1$.  Therefore, if $\gamma_1$ is large enough, this guarantees a negative sum, which gives the result.  In fact, we can say more:  for each $m$, the condition on $\gamma_1$ in terms of the other $\gamma_k$ is linear, and therefore to guarantee stability we need to show that
\begin{equation*}
  \max_{m=1,\dots, n-1} \sum_{k=1}^K \gamma_k \beta_{k,m,n} <0,  \quad\mbox{or,}\quad 
  \gamma_1 > \min_{m=1,\dots, n-1}\left(-\frac{1}{\beta_{1,m,n}} \sum_{k=2}^K \gamma_k \beta_{k,m,n}\right).
\end{equation*}
Since the right-hand side is a finite minimum of linear functions, it is piecewise linear. 

\begin{remark}
  In the special case of $K=2$, this reduces to $\gamma_1 > \rho^*(n) \gamma_2$, where
\begin{equation*}
  \rho^*(n) =  \min\left(0,\min_{m=1,\dots, n-1} \left(-\frac{\beta_{2,m,n}}{\beta_{1,m,n}}\right)\right)
\end{equation*}
In particular, when $n$ is even this is exactly zero, but for $n$ odd this is always positive (since at least one of the $\beta_{2,m,n}$ is negative).  In particular, we can see that when $n = 2k+1$ is odd, this expression is maximized at $m=k$, and for $k$ large we find that 
\begin{equation*}
  -\frac{\beta_{2,k,2k+1}}{\beta_{1,k,2k+1}} = \frac{\pi^2}{4k^2} - \frac{\pi^3}{4k^3} + O(k^{-4}),
\end{equation*}
so that $\rho^*(n) \approx (\pi/n)^2$ for $n$ odd.
\end{remark}

\qed

{\bf Proof of~\ref{thm:alphastability}.} 
Let us define two integrals:
\begin{align}
	I_m(\alpha) &:= \int_0^{2\pi\alpha} \cos(x) (\cos(mx)-1)\,dx,\label{eq:defofIm}\\
  	J_m(\alpha) &:= \int_{0}^{2\pi\alpha} \left(\cos((m+1)x) + \cos(mx)-\cos(x)-1\right) \,dx.
\end{align} 
If we  discretize by $x_k = 2\pi k/n$, $\Delta x = 2\pi/n$, then $\lambda_{1,m}$ (resp.~$\mu_{1,m}$) is a scalar multiple of the (left) Riemann sum of $I_m$ (resp.~$J_m$).  In particular, if $I_m(\alpha)<0$, then this implies that $\lambda_{1,m}$ is negative on the $\alpha$-sequence for $n$ sufficiently large, and similarly for $\mu_{1,m}$.  We will show that there exists $\alpha^*$ such that for all $0<\alpha<\alpha^*$, $I_m(\alpha),J_m(\alpha)<0$ for all $m$, and this will establish the result.

Let us first consider 
\begin{equation*}
  I_1(\alpha) = \pi  \alpha-\sin (2 \pi  \alpha)+\frac{1}{4} \sin (4 \pi \alpha).
\end{equation*}
The Taylor series expansion shows that $I_1(\alpha) < 0$ for $\alpha$ small and positive, and numerically we find that the first positive root of $I_1(\cdot)$ is at $\alpha^* \approx 0.340461$.  Now consider $m>1$.  We have
\begin{equation*}
  I_m(\alpha) = -\sin(2\pi \alpha) + \frac{m}{m^2-1} \cos (2 \pi  \alpha ) \sin (2 \pi m \alpha )- \frac{1}{m^2-1}\sin (2 \pi  \alpha ) \cos (2 \pi m \alpha ).
\end{equation*}
We break this into two cases:  $m$ even and $m$ odd.  If $m$ is even, we observe that $I_m(\alpha) = I_m(1/2-\alpha)$.  Since the integrand in~\eqref{eq:defofIm} is negative for all $x\in(0,1/4)$, this means that $I_m(\alpha)$ is monotone decreasing on $(0,1/4)$.  Using the symmetry, this means that $I_m(\alpha) \le 0$ for all $\alpha(0,1/2)$, and certainly for all $\alpha\in(0,\alpha^*)$.

If $m$ is odd, it is slightly more complicated.  Let us write $m=2k+1$. We still have that $I_{2k+1}(\alpha)$ is monotone decreasing on $\alpha\in(0,1/4)$, and we can check that $I_{2k+1}(1/4) = -1$.  Note that $1/2(I_{2k+1}(x)+I_{2k+1}(1/2-x)) = -\sin(2\pi\alpha)$ and 
\begin{equation*}
  \frac12 (I_{2k+1}(x)+I_{2k+1}(1/2-x)) = \frac{1}{4k} \sin(4k\pi x) + \frac{1}{4(k+1)}\sin(4(k+1)\pi x)
\end{equation*}
and thus
\begin{equation*}
  \frac12 \av{I_{2k+1}(x)+I_{2k+1}(1/2-x)} \le \frac{2k+1}{4(k^2+k)}.
\end{equation*}
This is decreasing in $k$, so the worst case is $k=1$, giving $3/8$.  Since $-\sin(2\pi\alpha^*)\approx -0.844328$, this means that $I_{2k+1}(\alpha) <0$ for all $0<\alpha<\alpha^*$.

Now we consider $J_m(\alpha)$.  Notice that if $m=0,-1$, then $J_m(\alpha) \equiv 0$, and otherwise
\begin{equation*}
J_m(\alpha) = \frac{\sin(2\pi\alpha(m+1))}{m+1} + \frac{\sin(2\pi \alpha m)}m - \sin(2\pi \a) - 2\pi \alpha, \quad m\neq 0,-1.
\end{equation*}
Using the inequality 
\begin{equation*}
\sin(x) \le x, \quad \sin(x) < x, x\neq 0,
\end{equation*}
gives
\begin{equation}\label{eq:BD}
\frac{\sin(2\pi\alpha m)}m -2\pi \alpha < 0
\end{equation}
for all $m\neq 0$.  Let us now show that if $\av{q}>1$, and $\beta>0$ is small enough, then 
\begin{equation}\label{eq:ACprep}
\sin(q \beta) \le q \sin(\beta).
\end{equation}
Fix $\beta$ and consider each side as a function of $q$.  (Since both sides are odd in $q$, we can consider only $q>0$.)  The left-hand side of~\eqref{eq:ACprep} is linear, and the right-hand side is concave down on the set $q\in[0,\pi/\beta]$, but the functions are equal at $q=0,1$.  Therefore~\eqref{eq:ACprep} holds for $q\in[1,\pi/\beta]$. Since the left-hand side is negative in $[\pi/\beta,2\pi/\beta]$, then the inequality also holds here.  Finally, note that the left-hand side is uniformly bounded above by $1$, so that if the right-hand side is larger than one for all $q>2\pi/\beta$, then the inequality holds there.  This is equivalent to saying that $\sin(\beta) > \beta/2\pi$, which holds for all $0 < \beta < \beta^* \approx 2.6978$.  

If we write $q=m+1, \beta = 2\pi\alpha$, then~\eqref{eq:ACprep} becomes 
\begin{equation}\label{eq:AC}
\sin(2\pi\alpha(m+1)) \le (m+1)\sin(2\pi \alpha),
\end{equation}
which holds for all $0 < \alpha <  \beta^*/(2\pi) \approx 0.429368$ and $m\not\in(-2,0)$.  Since $\beta^*/(2\pi) > \alpha^*$, this expression is negative on $(0,\alpha^*)$.

\qed

{\bf Proof of Proposition~\ref{prop:single}.}
Recall that $X_q = T^q$, and that we are considering only $f(x) =x$.  Comparing to Proposition~\ref{prop:linear}, we have
\begin{equation}\label{eq:TQQT}
  L_{X,ij} Q = T^{j-i}Q + QT^{i-j},
\end{equation}
so that the dependence of $L$ on $i,j$ depends solely on $i-j$.  Moreover, exchanging $i,j$ replaces $T$ with $T^*=T^{-1}$ in~\eqref{eq:TQQT}.  

Let us choose as a basis for $\mf{so}(d)$ the basis vectors $M_{ij}$, $i<j$, where $M_{ij}$ has a $1$ in position $(i,j)$, a $-1$ in position $(j,i)$, and is zero otherwise.  Then a direct computation shows that
\begin{align*}
  T M_{12} + M_{12}T^* &= \cos\theta M_{12},\\
  T M_{1q} + M_{1q}T^* &= \frac12(1+\cos\theta) M_{1q} + \frac12\sin\theta M_{2q},\quad q>2,\\
  T M_{2q} + M_{2q}T^* &= \frac12(1+\cos\theta) M_{2q} - \frac12\sin\theta M_{1q},\quad q>2,
\end{align*}
and for all other $M_{ij}$ not listed above, we have $TM_{ij}  + M_{ij}T^* = M_{ij}$.  In particular, if we consider the operator $L\colon\mf{so}(d)\to\mf{so}(d)$ given by $\mf LQ=TQ+QT^*$, and choose as ordered basis for $\mf{so}(d)$ the matrices $(M_{12},M_{13},M_{23},M_{14},M_{24},\dots, M_{34}, \dots)$, then the matrix representation of $\mf L$ in this basis is given by the matrix
\begin{equation}\label{eq:genrep}
  Q(\theta) = (\cos\theta) \oplus \left(\begin{array}{cc}\frac12(1+\cos\theta)& - \frac12\sin\theta\\\frac12\sin\theta&\frac12(1+\cos\theta)\end{array}\right)^{\oplus(d-2)}\oplus I_{(d-2)(d-3)/2}.
\end{equation}
Since $T$ is a twist by angle $\theta$,  replacing $T$ with $T^p$ just replaces $\theta$ with $p\theta$ in~\eqref{eq:genrep}.  What this means is that if we look at the entry $Q_{\a\b}(\theta)$ in~\eqref{eq:genrep}, this will allow us to generate the $(\a,\b)$ block of $J_Y$ in~\eqref{eq:defofJblock} as follows: the $(i,j)$th entry of $J_{\a\b}$ is given by 
\begin{equation*}
  (J_{\a\b})_{ij} = \gamma_{ij} Q_{\a\b}((j-i)\theta).
\end{equation*}
  Moreover, since $\Gamma$ is circulant, we have $\gamma_{ij} = \gamma_{\av{i-j}}$, and we can more simply write
\begin{equation*}
  (J_{\a\b})_{i,i\pm k} = \gamma_k Q_{\a\b}(\pm k\theta).
\end{equation*}
Note that if $Q_{\a\b}$ is an even function of $\theta$, this implies that $J_{kl}$ is symmetric, and if $Q_{kl}$ is an odd function of $\theta$, then $J_{kl}$ is skew-symmetric.

More compactly, let us the matrix $\AG f \theta$ as the $n\times n$ matrix with entries:
\begin{equation*}
  \left(\AG f\theta\right)_{ij} = \begin{cases} \gamma_k f(k\theta), & j=i+k,\\ -\sum_{\ell\neq 0}\gamma_\ell f(\ell\theta),&j=i.\end{cases}
\end{equation*}
and then it follows that the Jacobian at $X$ is given by the matrix
\begin{equation}\label{eq:JT}
  J_{X} = \ag(\cos\theta) \oplus \frac12 \left(\begin{array}{cc}\ag(1+\cos\theta)& - \ag(\sin\theta)\\\ag(\sin\theta)&\ag(1+\cos\theta)\end{array}\right)^{\oplus(d-2)}\oplus \ag(1)^{\oplus{(d-2)(d-3)/2}}
\end{equation}
Recalling~\eqref{eq:defofeigs}, if we can show that $\lambda_\ell$, $\mu_\ell$, $\nu_\ell$ are the eigenvalues of the three matrices appearing in~\eqref{eq:JT}, then we are done.  To establish this, we note that the first and third matrices are circulant, and while the second matrix is not quite ``block circulant with circulant blocks'' (BCCB), we can attack it in a similar fashion to that used for BCCB matrices~\cite{Tee.07, Gray.06, Davis.book}.
Using Lemma~\ref{lem:A+iB}, we see that the eigenvalues of the middle matrix are the same as the eigenvalues of 
\begin{equation*}
  \frac12(\ag(1+\cos\theta) + \mi (\ag\sin\theta)) = \frac12 \ag(1+e^{\mi \theta}).
\end{equation*}
Note that this last matrix is circulant.  Choose $\zeta$ to be an $n$th root of unity, and  define the vector $v$ such that $v_i  = \zeta^i$.  Writing $E = (1/2)\ag(1+e^{\mi \theta})$, we have
\begin{align*}
  (Ev)_i &= 
  \sum_{k=1}^K E_{i,i+k}v_{i+k} +  \sum_{k=1}^K E_{i,i-k}v_{i-k} + E_{ii}v_i\\
  &= \sum_{k=1}^K\frac{\gamma_k}2(1+e^{\mi k\theta})\zeta^{i+k}+\sum_{k=1}^K\frac{\gamma_k}2(1+e^{-\mi k\theta})\zeta^{i-k} + \left(-\sum_{k\neq 0}\frac{\gamma_k}2 (1+e^{\mi k\theta})\right)\zeta^i\\
  &= \zeta^i\left(\sum_{k=1}^K\frac{\gamma_k}2 (1+e^{\mi k\theta})(\zeta^{k}-1)+\sum_{k=1}^K\frac{\gamma_k}2 (1+e^{-\mi k\theta})(\zeta^{-k}-1)\right),
\end{align*}
and since $v_i =\zeta^i$, the expression in parentheses is an eigenvalue of $E$.
Let us now assume that $\theta = 2\pi l/n$, and choose $\zeta = \exp(\mi 2\pi m/n)$.  (Note here that $l$ is fixed by the solution $X$, but $m$ ranges over $0,\dots,n-1$.) Then we have
\begin{align*}
  &\sum_{k=1}^K\frac{\gamma_k}2 (1+e^{\mi k\theta})(\zeta^{k}-1)+\sum_{k=1}^K\frac{\gamma_k}2 (1+e^{-\mi k\theta})(\zeta^{-k}-1)\\
  &\quad= \sum_{k=1}^K\frac{\gamma_k}2\left(e^{\mi \frac{2\pi}n k(\ell+m)} + e^{\mi \frac{2\pi}n km}- e^{\mi \frac{2\pi}n k\ell} -1\right) + c.c.\\
  &\quad= \sum_{k\neq 0 }\gamma_k \left\{ \bigcos{k(\ell+m)}+\bigcos{km} - \bigcos{k\ell}-1\right\}.
\end{align*}
This is the definition of $\mu_{\ell,m}$ in~\eqref{eq:defofeigs}.  By~\eqref{eq:JT} there are $d-2$ copies of the inner matrix, and by Lemma~\eqref{lem:A+iB}, these eigenvalues are doubled, so each $\mu_{\ell,m}$ appears with multiplicity $2(d-2)$.

The formulas for $\lambda_{\ell,m}, \nu_{\ell,m}$ are similar but a bit simpler.  Again noting that these are circulant, the same approach shows that the eigenvalues of $\ag(\cos\theta)$ are $\lambda_{\ell,m}$ and the eigenvalues of $\ag(1)$ are $\nu_{\ell,m}$.

\qed

{\bf Proof of Propositions~\ref{prop:double} and~\ref{prop:ho}.}
  The proof will be similar to that given in Proposition~\ref{prop:single}.  The idea is as follows:  let us assume that $T = \twist(\theta_1,\theta_2,0,0,\dots,0)$.  Let us write $\mc L Q = \frac12(TQ+QT^*)$ as above, then we have
  \begin{align*}
    \mc L M_{12} &= \cos\theta_1 M_{12},\quad \mc L M_{34} = \cos\theta_2 M_{34},\\
    \mc L \left(\begin{array}{c}M_{1q}\\M_{2q}\end{array}\right) &= \frac12\left(\begin{array}{cc}1+\cos\theta_1& -\sin\theta_1\\\sin\theta_1&1+\cos\theta_1\end{array}\right)\left(\begin{array}{c}M_{1q}\\M_{2q}\end{array}\right), \quad q>4,\\  
    \mc L \left(\begin{array}{c}M_{3q}\\M_{4q}\end{array}\right) &= \frac12\left(\begin{array}{cc}1+\cos\theta_2& - \sin\theta_2\\\sin\theta_2&1+\cos\theta_2\end{array}\right)\left(\begin{array}{c}M_{3q}\\M_{4q}\end{array}\right), \quad q>4,\\
    \mc L\left(\begin{array}{c}M_{13}\\M_{24}\\M_{23}\\M_{24}\end{array}\right) &=  \frac12\left(\begin{array}{cccc}\cos\theta_1+\cos\theta_2& -\sin\theta_2&-\sin\theta_1&0\\\sin\theta_2&\cos\theta_1+\cos\theta_2&0&-\sin\theta_1\\\sin\theta_1&0&\cos\theta_1+\cos\theta_2 &-\sin\theta_2\\0&\sin\theta_1&\sin\theta_2&\cos\theta_1+\cos\theta_2 \end{array}\right)\left(\begin{array}{c}M_{13}\\M_{24}\\M_{23}\\M_{24}\end{array}\right),
  \end{align*}
  and for all $M_{ij}$ not listed above, $\mc L M_{ij} = M_{ij}$.
Form this it follows that the Jacobian at $X$ is given by the matrix
\begin{equation}\label{eq:JT2}
\begin{split}
  J_{X} = \ag(\cos\theta_1) \oplus\ag(\cos\theta_2) \oplus \frac12 \left(\begin{array}{cc}\ag(1+\cos\theta_1)& - \ag(\sin\theta_1)\\\ag(\sin\theta_1)&\ag(1+\cos\theta_1)\end{array}\right)^{\oplus(d-4)}\\\oplus \frac12 \left(\begin{array}{cc}\ag(1+\cos\theta_2)& - \ag(\sin\theta_2)\\\ag(\sin\theta_2)&\ag(1+\cos\theta_2)\end{array}\right)^{\oplus(d-4)}\oplus \ag(1)^{\oplus{(d-4)(d-5)/2}}\oplus\\\oplus\frac12\left(\begin{array}{cccc}\ag(\cos\theta_1+\cos\theta_2)& -\ag(\sin\theta_2)&-\ag(\sin\theta_1)&0\\\ag(\sin\theta_2)&\ag(\cos\theta_1+\cos\theta_2)&0&-\ag(\sin\theta_1)\\\ag(\sin\theta_1)&0&\ag(\cos\theta_1+\cos\theta_2) &-\ag(\sin\theta_2)\\0&\ag(\sin\theta_1)&\ag(\sin\theta_2)&\ag(\cos\theta_1+\cos\theta_2) \end{array}\right).
  \end{split}
\end{equation}

All of the formulas for the eigenvalues except for $\kappa$ are derived exactly the way as they are in the proof of Proposition~\ref{prop:single} --- the only difference in all of these is that instead of a single $\theta$ we now have $\theta_1$ or $\theta_2$.  The matrix that is new is the last one.  Putting Lemma~\ref{lem:ABC} together with the last matrix in~\eqref{eq:JT2}, we see that the remaining $\kappa_{\ell_1,\ell_2,m}^\pm$ are the eigenvalues of
\begin{equation*}
  \frac12\ag(e^{\mi\theta_1} + e^{\pm\mi\theta_2}).
\end{equation*}

Note that this last matrix is circulant.  Choose $\zeta$ to be an $n$th root of unity, and  define the vector $v$ such that $v_i  = \zeta^i$.  Writing $E = (1/2)\ag(e^{\mi\theta_1} + e^{\pm\mi\theta_2})$, we have
\begin{align*}
  (Ev)_i &= 
  \sum_{k=1}^K E_{i,i+k}v_{i+k} +  \sum_{k=1}^K E_{i,i-k}v_{i-k} + E_{ii}v_i\\
  &= \sum_{k=1}^K\frac{\gamma_k}2(e^{\mi k\theta_1}+e^{\pm\mi k\theta_2})\zeta^{i+k}+\sum_{k=1}^K\frac{\gamma_k}2(e^{-\mi k\theta_1}+e^{\mp\mi k\theta_2})\zeta^{i-k} + \left(-\sum_{k\neq 0}\frac{\gamma_k}2 (e^{\mi k\theta_1}+e^{\pm\mi k\theta_2})\right)\zeta^i\\
  &= \zeta^i\left(\sum_{k=1}^K\frac{\gamma_k}2 (e^{\mi k\theta_1}+e^{\pm\mi k\theta_2})(\zeta^{k}-1)+c.c.\right),
\end{align*}
and since $v_i =\zeta^i$, the expression in parentheses is an eigenvalue of $E$.  If we write $\theta_a = 2\pi \ell_a/n$, and choose $\zeta = \exp(\mi 2\pi m/n)$, then
\begin{align*}
  &\sum_{k=1}^K\frac{\gamma_k}2 (e^{\mi k\theta_1}+e^{\pm\mi k\theta_2})(\zeta^{k}-1)+c.c.)\\
  &\quad= \sum_{k=1}^K\frac{\gamma_k}2\left(e^{\mi \frac{2\pi}n k(\ell_1+m)} + e^{\mi \frac{2\pi}n k(\pm \ell_2+m)}- e^{\mi \frac{2\pi}n k\ell_1} -e^{\mi \frac{2\pi}n k\ell_2}\right) + c.c.\\
  &\quad= \sum_{k=1}^K \gamma_k \left\{\bigcos{k(\ell_1+m)}+\bigcos{k(\pm \ell_2+m)} - \bigcos{k \ell_1} - \bigcos{\pm k\ell_2}\right\}.
\end{align*}
This is the definition of $\kappa^\pm_{\ell_1,\ell_2,m}$ in~\eqref{eq:defofeigs2}.  By~\eqref{eq:JT2} there are $d-2$ copies of the inner matrix, and by Lemma~\eqref{lem:A+iB}, these eigenvalues are doubled, so each $\mu_{\ell_a,m}$ appears with multiplicity $2(d-2)$.

More generally, let us assume that $T = \twist(\theta_1,\theta_2,\dots)$ where the unwritten angles may or may not be zero.  The proof above implies that the $4\times4$ matrix that appears  in~\eqref{eq:JT2} also appears as a term in the direct sum of the Jacobian for the higher-order rotation, and and such contains the eigenvalues $\kappa^\pm_{\ell_1,\ell_2,m}$ as well. 

\qed

\end{document}